\documentclass{amsart}
\usepackage{amsxtra, microtype, latexsym,wasysym}

\usepackage{etex}
\usepackage[T1]{fontenc}
\usepackage[english]{babel}
\usepackage{mathtools}
\usepackage[margin=1.45in]{geometry}
\usepackage[all]{xy}
\usepackage{textcomp}
\usepackage{enumerate}
\usepackage{enumitem}
\usepackage{microtype}
\usepackage{lmodern}
\usepackage{graphicx}
\usepackage[numbers]{natbib}

\usepackage[text]{esdiff}
\usepackage{scalerel}

\usepackage{amsfonts}
\usepackage{amsmath}
\usepackage{amstext}
\usepackage{amssymb}
\usepackage{amsthm}
\usepackage{stmaryrd, mathrsfs}
\usepackage{cancel}

\usepackage{tikz}
\usepackage{tikz-cd}
\usepackage{units}

\usepackage{hyperref}
\usepackage{xcolor}
\usepackage{scalerel}
\usepackage{stackengine}
\usepackage{stackrel}

\newcommand*{\DMbra}[2]{\llbracket#1,
#2\rrbracket}
\newcommand*{\bigDMbra}[2]{\bigl\llbracket#1,
#2\bigr\rrbracket}
\newcommand{\pair}[2]{\langle#1,
#2\rangle}
\newcommand{\bigpair}[2]{\bigl\langle#1,
#2\bigr\rangle}
\newcommand{\Bigpair}[2]{\Bigl\langle#1,
#2\Bigr\rangle}

\stackMath

\newcommand\corelinear[1]{%
\savestack{\tmpbox}{\stretchto{%
  \scaleto{%
    \scalerel*[\widthof{\ensuremath{#1}}]{\kern-.6pt\sim\kern-.6pt}%
    {\rule[-\textheight/2]{1ex}{\textheight}}%WIDTH-LIMITED BIG WEDGE
  }{\textheight}% 
}{0.5ex}}%
\stackon[1pt]{#1}{\tmpbox}%
}
\DeclareMathOperator {\pr}    {pr}
\DeclareMathOperator {\grap} {graph}
\DeclareMathOperator {\Hom}   {Hom}
\DeclareMathOperator {\id}    {id}
\DeclareMathOperator {\End}    {End}

\newcommand{\dr}{\mathbf{d}}
\newcommand{\llb}{\llbracket} 
\newcommand{\rrb}{\rrbracket}
\newcommand{\eps}   {\varepsilon}
\newcommand{\half}   {\frac{1}{2}}
\newcommand{\TT}       {\mathbb{T}}
\newcommand{\Jac}       {\mathbf{Jac}}
\newcommand{\R}       {\mathbb{R}}
\newcommand{\mx}{\mathfrak{X}} 
\newcommand{\C}{\mathbb{C}}
\newcommand{\AAA}{\mathbb{A}}
\newcommand{\dd}{\mathbf{d}}
\newcommand{\bas}{{\scriptstyle\mathrm{bas}}}
\newcommand{\JJ}{\mathcal{J}}
\newcommand{\LL}{\mathcal{L}}
\newcommand{\E}{\mathbb{E}}
\newcommand{\XX}{\mathfrak{X}}
\newcommand{\slift}{\sigma^\Delta}

\newcommand{\an}{\arrowvert_}

\newcommand{\ldr}[1]{{{\pounds}}_{#1}}

\makeatletter
\@namedef{subjclassname@2020}{\textup{2020} Mathematics Subject Classification}
\makeatother

\newtheorem{thm}{Theorem}[section]
\newtheorem{mydef}[thm]{Definition}
\newtheorem{lemma}[thm]{Lemma}
\newtheorem{cor}[thm]{Corollary}
\newtheorem{prop}[thm]{Proposition}
\newtheorem{rmk}[thm]{Remark}
\newtheorem{example}[thm]{Example}

\begin{document}
%%%%%%%%%%%%%%%%%%%%%%%%%%%%%%%%%%%%%%%%%%%%%%%%%%%%%%%%%%%%%%%%%%%%%%%%%%%
%%%%%%%%%%%%%%%%%%%%%%    Title    %%%%%%%%%%%%%%%%%%%%%%%%%%%%%%%%%%%%%%%%
\title{Linear generalised complex structures}

%%% author one information

\author{M. Heuer} \address{Department of Mathematics, University of Hamburg, Germany} \email{malte.heuer@uni-hamburg.de}
\author{M. Jotz Lean} \address{Mathematisches Institut, Georg-August Universit\"at G\"ottingen, Germany}
\email{madeleine.jotz-lean@mathematik.uni-goettingen.de}

\begin{abstract}
  This paper studies linear generalised complex structures over vector
  bundles, as a generalised geometry version of holomorphic vector
  bundles. In an adapted linear splitting, a linear generalised
  complex structure on a vector bundle $E\to M$ is equivalent to a
  $\mathbb C$-multiplication $j$ in the fibers of $TM\oplus E^*$ and
  a $\C$-Lie algebroid structure on $TM\oplus E^*$. 

  Generalised complex Lie algebroids (or Glanon algebroids) are then
  studied in this context, and expressed as a pair of complex
  conjugated Lie bialgebroids.

\end{abstract}

\maketitle

\tableofcontents
\section{Introduction}

\addtocontents{toc}{\protect\setcounter{tocdepth}{1}}

This paper studies linear generalised complex structures on vector
bundles and on Lie algebroids.  Generalised complex geometry was
introduced by Hitchin in \cite{Hitchin03} as a unification of
symplectic and complex geometry.  It was further developed by
Gualtieri in his thesis \cite{Gualtieri04,Gualtieri11}. Since they
simultaneously unify symplectic and complex structures, generalised
complex structures have been studied for their relation to T-duality
-- a concept arising in string theory -- by Cavalcanti and Gualtieri
in \cite{CaGu10}.
Gualtieri also defined generalised K{\"a}hler structures in
\cite{Gualtieri04,Gualtieri14}. These have been studied for example in
\cite{Hitchin06} and in \cite{BuCaGu08}.

The relation between generalised complex geometry and Lie algebroids
and Lie groupoids was first studied by Crainic in \cite{Crainic11},
and generalised complex structures \emph{on} Lie groupoids and Lie
algebroids were studied in \cite{JoStXu16}. In particular,
\cite{JoStXu16} proves that multiplicative generalised complex
structures on source simply connected Lie groupoids are equivalent to
generalised complex integrable Lie algebroids.  This paper studies in
more detail the obtained generalised complex Lie algebroids (or
\emph{Glanon algebroids}), and in particular the underlying special
case of generalised complex vector bundles.

\medskip
      
The notion of generalised holomorphic bundles has been introduced by
Gualtieri \cite{Gualtieri11,Gualtieri10} as a complex vector bundle
$E$ over a generalised complex manifold $M$, equipped with a flat
$L$-connection $\bar\partial\colon \Gamma(E)\to \Gamma(L^*\otimes E)$,
where $L$ is the $+i$-eigenbundle of the generalised complex structure
on $M$. This is in analogy to how a holomorphic vector bundle
structure on $E$ over a complex manifold $M$ is equivalent to a flat
Dolbeault operator $\bar\partial_E$.  However, while a holomorphic
vector bundle $E\to M$ amounts to a linear complex structure on the
smooth manifold $E$, in Gualtieri's definition of a generalised
holomorphic vector bundle the manifold $E$ itself does not carry a
generalised complex structure.

This paper instead adopts the point of view that a ``generalised
holomorphic vector bundle'' should be a smooth vector bundle equipped
with a \emph{linear} generalised complex structure, and explores the
property of such an object -- here, it is now the manifold $M$ which
does not automatically inherit a generalised complex structure. Note
that symplectic vector bundles, i.e.~smooth vector bundles equipped
with a linear symplectic form, and Poisson holomorphic vector bundles
are natural examples of this notion of generalised holomorphic vector
bundle.

The following sections describe holomorphic vector bundles as linear
complex structures on vector bundles, and explain how the obtained
infinitesimal structures in this setting can be recovered in the more
general context of linear generalised complex structures.

\subsection*{Holomorphic vector bundles and linear complex structures}
If $q\colon E\to M$ is a holomorphic vector bundle over a complex
manifold $M$, then $TE\to E$ and $TM\to M$ are also holomorphic vector
bundles.
The complex structure $J_E\colon TE\to TE$ is then a vector bundle 
morphism over the complex structure $J_M\colon TM\to TM$. It is
easy to see (see Section \ref{new_complex_subsection}) that
this defines a morphism of double vector bundles
\begin{equation}\label{diag_JE}{\footnotesize
		\begin{tikzcd}
		TE \ar[rr,"J_E"]\ar[rd]\ar[dd] & & TE \ar[dd]\ar[rd] & \\
				& E  \ar[rr,"\id_E", near start, crossing over] & & E \ar[dd] \\
		TM \ar[rr,"J_M", near start]\ar[rd] & & TM\ar[rd] & \\
				& M\ar[rr,"\id_M"] \ar[uu, crossing over, leftarrow] &   & M
		\end{tikzcd} 
	\,}
	\end{equation}
		with core morphism $j_E\colon E\to E$,
        the multiplication by $i$ in the fibers of $E$.
        
        A linear complex structure $J_E$ as in \eqref{diag_JE} on a smooth
        vector bundle $E$ is in fact equivalent to a holomorphic structure 
        on $E$. This follows from the corresponding, more general result about 
        holomorphic groupoids of \cite{LaStXu09,BuDr19}.
        For the convenience of the reader, and since this approach to 
       holomorphic vector bundles motivates this paper's study of
        generalised complex vector bundles, a more direct proof of this
       correspondence is carried out in detail in Section
       \ref{new_complex_section}.
     
        \subsection*{Linear generalised complex
              structures}

        This paper studies the generalisation of this description of
        holomorphic vector bundles to vector bundles endowed with a
        \emph{linear generalised complex structure}. Since
        the terminology of \emph{generalised holomorphic vector
          bundle} is already used in the literature for a different
        generalisation of holomorphic vector bundles, here vector bundles
        endowed with a linear generalised complex structure are simply
        called \emph{generalised complex vector bundles}.

        Let $E\to M$ be a smooth vector bundle.  The generalised
        tangent bundle $\TT E=TE\oplus T^*E$ is then a double vector
        bundle
\begin{equation*}
	\begin{tikzcd}
	TE\oplus T^*E \ar[r] \ar[d] & E \ar[d]\\
	TM\oplus E^* \ar[r]  & M
	\end{tikzcd}
\end{equation*}
with core $E\oplus T^*M$. The vector bundle $TE\oplus T^*E\to E$ is
naturally equipped with the standard Courant algebroid structure over
the manifold $E$.

\begin{mydef}\label{def_lin_gcs}
  A generalised complex structure $\JJ$ on a vector bundle $E\to M$ is
  called \textbf{linear} if
  $\mathcal{J}\colon TE\oplus T^*E\rightarrow TE\oplus T^*E$ is a
  morphism of double vector bundles over a side morphism
  $j\colon TM\oplus E^*\rightarrow TM\oplus E^*$ and with a core
  morphism $j_C\colon E\oplus T^*M \rightarrow E\oplus T^*M$.
	\begin{equation}\label{diag_J}{\footnotesize
		\begin{tikzcd}
                  TE\oplus T^*E \ar[rr,"\JJ"]\ar[rd]\ar[dd] & & TE\oplus T^*E \ar[dd]\ar[rd] & \\
                  & E  \ar[rr,"\id_E", near start, crossing over] & & E \ar[dd] \\
                  TM\oplus E^* \ar[rr,"j", near start]\ar[rd] & & TM\oplus E^*\ar[rd] & \\
                  & M\ar[rr,"\id_M"] \ar[uu, crossing over, leftarrow]
                  & & M
		\end{tikzcd} 
	\,.}
      \end{equation}
      \end{mydef}
      Consider first a linear generalised complex structure on a
      vector space $V$ (i.e.~on a vector bundle over a point). In this
      case the tangent and cotangent bundle are canonically split,
      $TV\simeq V\times V$ and $T^*V\simeq V\times V^*$. The linearity
      condition on the generalised complex structure $\JJ$ is
      equivalent to $\JJ$ being determined by the maps in the fibres,
      that is by the side morphism $j_{V^*}\colon V^*\to V^*$ and the
      core morphism $j_C\colon V\to V$. These have to be in negative
      duality to each other, that is $j=-j_C^t$, since the generalised
      complex structure is orthogonal with respect to the canonical
      pairing. Therefore, a linear generalised complex structure on a
      vector space in this sense is equivalent to the choice of an
      ordinary complex structure in the vector space.

      Back to the general case, this paper shows that after the choice
      of an adequate linear splitting of $TE\oplus T^*E$, the
      generalised complex structure is equivalent to a special complex
      Lie algebroid structure on $TM\oplus E^*$, as in the following
      definition -- the bundle $TM\oplus E^*$ is seen as a complex
      vector bundle with $j\colon TM\oplus E^* \to TM\oplus E^*$ as
      its multiplication by $i$.

      \begin{mydef}\label{def_qrca}
        Let $Q\to M$ be a vector bundle with complex
        fibers, hence with a vector bundle morphism $j\colon Q\to Q$
        such that $j^2=-\id_Q$. A \textbf{complex Lie algebroid structure}
        on $(Q,j)$ is
         a $\C$-bilinear Lie algebra bracket $[\cdot\,,\cdot]$ on sections of $Q$ and a morphism
         $\lambda\colon Q\to T_\C M$ of complex vector bundles that anchors the bracket:
              $ [q_1,fq_2]=\lambda(q_1)(f)q_2+f[q_1,q_2]$
              for all $f\in C^\infty(M,\C)$ and $q_1,q_2\in\Gamma(Q)$.

          A complex Lie algebroid
            $(Q\to M, j,\lambda, [\cdot\,,\cdot])$ is
            \textbf{quasi-real} if there exists
            \begin{enumerate}
            \item
              a real vector bundle morphism $\rho\colon Q\to TM$ such that
              $\lambda=\rho_j\colon Q\to T_\C M$ defined by
               \begin{equation}\label{special_complex_anchor}
                 \rho_j (q):=\half(\rho(q)\otimes 1-\rho(jq)\otimes i)
                 \end{equation}
                 for all $q\in\Gamma(Q)$, and
              \item  a dull bracket
            $\llb\cdot\,,\cdot\rrb$ on $(Q,\rho)$ such that the
            complexification
            $(Q_\C\to M,\rho_\C,\llb\cdot\,,\cdot\rrb_\C)$ restricts to
            $(Q^{1,0}\to M, \rho_{1,0},\llb\cdot\,,\cdot\rrb_{1,0})$ on
            the $i$-eigenspace of $j_\C$, which coincides with the
            complex Lie algebroid
            $(Q\to M,\rho,j, [\cdot\,,\cdot])$
            via the canonical isomorphism
            \[Q\to Q^{1,0},\qquad 
                q\mapsto \half(q\otimes 1- j(q)\otimes i).
              \]
              \end{enumerate}
            \end{mydef}

            The definition above of a complex algebroid follows the
            one in \cite{Weinstein07}. 
            The following theorem is
            the main result of this paper.
      \begin{thm}\label{main}
        Let $E\to M$ be a smooth vector bundle.  A linear generalised
        complex structure on $E$ with side
        $j\colon TM\oplus E^*\to TM\oplus E^*$ is equivalent to a
        quasi-real complex Lie algebroid structure on
        $(TM\oplus E^*,j)$ with anchor $\pr_{TM, j}\colon TM\oplus E^*\to T_\C M$,
        $\nu\mapsto \half(\pr_{TM}\nu\otimes 1-\pr_{TM}(j\nu)\otimes i)$.
        \end{thm}
        In other words, a linear generalised complex structure with
        side $j$ on a vector bundle $E$ is equivalent to a special
        complex Lie algebroid structure on the complex vector bundle
        $(TM\oplus E^*\to M,j)$. The equivalence in Theorem \ref{main}
        is of course the most important part of the statement. It is
        explained along the introduction of the necessary tools: via
        this equivalence, a quasi-real Lie algebroid structure on
        $(TM\oplus E^*,j)$ is sent to a generalised complex structure
        $\mathcal J\colon TE\oplus T^*E\to TE\oplus T^*E$ such that
        any dull bracket on $TM\oplus E^*$ realising the complex Lie
        bracket on $(TM\oplus E^*, j)$ is \emph{adapted to} $\mathcal J$.

        In the case of a holomorphic vector bundle, the complex Lie
        algebroid found in Theorem \ref{main} is simply
        $T^{1,0}M\oplus (E^{0,1})^*\to M$, with the bracket defined by
        the complex Lie algebroid bracket on $T^{1,0}M$ (since $M$ is
        a complex manifold), and the flat $T^{1,0}M$-connection on
        $E^{0,1}$ that is complex conjugated to the
        $\bar\partial$-operator
        $\bar\partial\colon\Gamma(T^{0,1}M)\times\Gamma(E^{1,0})\to\Gamma(E^{1,0})$. See
        Example \ref{hol_ex_cpla}.

        \medskip

        Section \ref{sec_gcs_VB_CA} extends the results of Section
        \ref{sec_gcs_vb} to the more general case of linear
        generalised complex structures in VB-Courant
        algebroids. Making use of the correspondence between
        VB-Courant algebroids and Lie 2-algebroids
        \cite{LiBland12,Jotz19b}, this leads after the choice of a
        linear splitting to a definition of generalised complex
        structures in split Lie 2-algebroids.

        \subsection*{Generalised complex Lie algebroids}
        In Section \ref{sec_gcLA} the vector bundle is equipped with
        the additional structure of a Lie algebroid, and the linear
        generalised complex structure is required to be compatible
        with the Lie algebroid structure. The obtained
        \emph{generalised complex Lie algebroids} were already studied
        in \cite{JoStXu16}, where they are called ``Glanon
        algebroids''.  The paper \cite{JoStXu16} gives a
        correspondence between \emph{multiplicative} generalised
        complex structures on Lie groupoids and compatible generalised
        complex structures on Lie algebroids. Hence in order to better
        understand generalised complex Lie groupoids it is useful to
        study generalised complex Lie algebroids in this sense. The
        goal of this section is a deeper study of the properties of
        generalised complex Lie algebroids, in the spirit of the study
        of holomorphic Lie algebroids done in \cite{LaStXu08}: that
        paper studies holomorphic Lie algebroids in detail and shows
        an equivalence between holomorphic Lie algebroid structures on
        $A\to M$ and linear holomorphic Poisson structures on the
        complex dual $\Hom_\C(A,\C)$. Moreover, it shows that a
        holomorphic Lie algebroid structure on $A$ is equivalent to a
        matched pair \cite{Mokri97,Mackenzie11} of the complex Lie
        algebroids $T^{0,1}M$ and $A^{1,0}$. Additionally, for a
        complex manifold $M$, the Lie algebroids $T^{1,0}M$ and
        $T^{0,1}M$ form a matched pair of complex Lie algebroids with
        matched sum $T_\C M$ and more generally, for a holomorphic Lie
        algebroid $A$ the Lie algebroids $A^{1,0}$ and $A^{0,1}$ form
        a matched pair with matched sum $A_\C$.

        \cite{JoStXu16} proves that a Poisson holomorphic Lie
        algebroid is equivalent to a holomorphic Lie bialgebroid.
        More generally, Section \ref{sec_gcLA} proves the following theorem:
\begin{thm}\label{main2}
  Let $A\to M$ be a Lie algebroid. Let $\mathcal J\colon TA\oplus T^*A\to TA\oplus T^*A$ be a linear generalised complex
  structure on $A$ with side $j\colon TM\oplus A^*\to TM\oplus A^*$.
  Then $(A,\mathcal J)$ is a Glanon algebroid 
  if and only if the quasi-real complex Lie
        algebroid structure on $(TM\oplus A^*, j)$ found in Theorem \ref{main} fits
        in a complex Lie bialgebroid $(TM\oplus A^*,K_-)$ over $M$.
        \end{thm}
         Here, the complex Lie algebroid
        $TM\oplus A^*$ is identified with $U_+$, the $i$-eigenspace of
        $j_{\mathbb C}$ in $(TM\oplus A^*)_{\mathbb C}$ equipped with
        the complexification of a dull bracket realising the one on
        $(TM\oplus A^*,j)$, as in (2) of Definition \ref{def_qrca}.
        The space $K_-$ is then the $i$-eigenspace of
        $(-j^t)_{\mathbb C}$, hence
        $K_-\simeq (A\oplus T^*M)_{\mathbb C}/K_+\simeq U_+^*$, since
        $K_+$ is the annihilator of $U_+$. The Lie algebroid structure
        on $A$ induces a \emph{degenerate} Courant algebroid structure
        on $A\oplus T^*M$ (see Section \ref{sec_deg_gc_ca}), the
        complexification of which has $K_-$ and $K_+$ as Dirac
        structures.

        The Drinfeld double Courant algebroid of the obtained complex
        Lie bialgebroid is isomorphic to a Courant algebroid $C_\pm$,
        which is obtained via a construction with complex $A$-Manin
        pairs from $A$ and $U_\pm$ (as in \cite{Jotz19a}).  In the
        initially studied special case of holomorphic Lie algebroids,
        where the generalised complex structure is induced by a
        complex structure, these Courant algebroids $C_\pm$ decompose
        as a direct sum of Courant algebroids,
        $C^{1,0}_T\oplus C^{0,1}_A$, or $C^{0,1}_T\oplus C^{1,0}_A$,
        respectively. The Courant algebroid structure in $C_\pm$ is
        then the same as the matched pair Courant algebroid structure
        on these bundles already given in \cite{GrSt14}.  This matched
        pair of Courant algebroids arises from the aforementioned
        matched pair of Lie algebroids $(T^{0,1}M, A^{1,0})$ of
        \cite{LaStXu08}.  The Courant algebroids $C_\pm$ therefore are
        the generalised version of this matched sum Courant algebroid
        in the special case of a holomorphic Lie algebroid.
        
        \subsection*{Methodology}
       The key to the results in this paper is the equivalence of 
        linear splittings of $TE\oplus T^*E$ with
        $TM\oplus E^*$-Dorfman connections on $E\oplus T^*M$, in a
        similar way as linear splittings of $TE$ are equivalent to
        linear $TM$-connections on $E$; see \cite{Jotz18a}. Section
        \ref{sec_bg_gentan} recalls this correspondence.  Proposition
        \ref{prop_adapted_DMC} establishes the existence of an adapted
        Dorfman connection to any linear generalised almost complex
        structure $\mathcal J$ on $E$. This is a Dorfman connection
        the horizontal sections of which are preserved by the
        generalised complex structure $\mathcal J$.

        This allows then a description of the properties of the
        generalised complex structures $\JJ$ in terms of this adapted
        Dorfman connection and the side morphism
        $j\colon TM\oplus E^*\to TM\oplus E^*$, see Theorem
        \ref{thm_lgcs_DMC}. Theorem \ref{main} follows immediately
        from the study of the integrability of the linear generalised
        complex structure in terms of its side morphism and an adapted
        Dorfman connection.

                \subsection*{Outline of the paper}

                Section \ref{sec_Background} recalls some necessary
                background on Courant algebroids and generalised
                complex structures, on linear splittings of VB-Courant
                algebroids and Dorfman connections and on morphisms of
                2-representations of Lie algebroids. Section
                \ref{new_complex_section} recalls the description of
                holomorphic vector bundles via linear complex
                structures on a real vector bundle.

                Section
                \ref{sec_gcs_vb} then studies linear generalised
                complex structures on vector bundles in detail.  For
                the convenience of the reader the basic definitions
                and properties of generalised complex structures are
                given as well in Section \ref{sec_bg_gcs}.  This
                section proves the existence of an adapted Dorfman
                connection and uses it to construct the complex Lie
                algebroid in Theorem \ref{main}.

                In Section \ref{sec_gcLA}, the vector bundle is
                endowed with the additional structure of a Lie
                algebroid and the linear generalised complex structure
                $\JJ$ is assumed to be compatible with the Lie
                algebroid structure. This is equivalent to $\JJ$
                defining a morphism of 2-representations. Some results
                of \cite{LaStXu08} are expanded in this more general framework.
                
                Finally, Section \ref{sec_gcs_VB_CA} studies the more
                general case of a linear generalised complex structure
                in a VB-Courant algebroid. Appendix \ref{app_comparison} compares
                for completeness this paper's adapted Dorfman
                connections with the adapted generalised connections
                in \cite{CoDa19}.

      \subsection*{Prerequisites and notation}

      All manifolds and vector bundles in this paper are smooth and
      real.  The reader is referred to Section 2.3 of \cite{Jotz18a}
      for the definition of a double vector bundle, their morphisms
      and induced core morphisms and for the necessary background on
      linear and core sections, and on their linear splittings and
      dualisations.  Section 2.3 of \cite{Jotz18a} recalls the
      definition of a VB-algebroid, and also the equivalence of
      $2$-term representations up to homotopy (called here
      \emph{$2$-representations} for short) with linear decompositions
      of VB-algebroids \cite{GrMe10}. The notation used here is the
      same as in \cite{Jotz18a}. In particular, a linear splitting of
      a double vector bundle $(D,A,B,M)$ is written
      $\Sigma\colon A\times_MB\to D$, and the corresponding horizontal
      lifts are then
      $\sigma:=\sigma_A\colon \Gamma(A)\to \Gamma_B^l(D)$ and
      $\sigma:=\sigma_B\colon \Gamma(B)\to \Gamma_A^l(D)$. The reader
      is invited to consult also \cite{Pradines77,Mackenzie05,GrMe10}
      for more details on double vector bundles.

       Vector bundle projections are written
      $q_E\colon E\to M$, and $p_M\colon TM\to M$ for tangent bundles.
      Given a section $\varepsilon$ of $E^*$, the map
      $\ell_\varepsilon\colon E\to \R$ is the linear function
      associated to it, i.e.~the function defined by
      $e_m\mapsto \langle \varepsilon(m), e_m\rangle$ for all
      $e_m\in E$.  The set of global sections of a vector bundle
      $E\to M$ is denoted by $\Gamma(E)$.      
      Elements $e\otimes (a+ib)$ of the complexification $E_\C$ maz be
      written as $ae+ibe$ if there is no confusion with a complex 
      structure of the vector bundle $E$ itself. The dual of a vector 
      bundle morphism $\varphi$ is written as $\varphi^t$, to avoid 
      confusion with pullbacks.
\subsection*{Acknowledgements}

The authors thank Thiago Drummond, Vicente Cort{\'e}s, Ping Xu,
Chenchang Zhu for their helpful comments and
suggestions.

\addtocontents{toc}{\protect\setcounter{tocdepth}{2}}

\section{Background}\label{sec_Background}

This section recalls basic notions and results, in particular on
linear splittings of the generalised tangent bundle of a vector
bundle \cite{Jotz18a}. 

\subsection{Courant algebroids and generalised complex structures}
\label{sec_bg_gcs}

Let $(\mathbb E\to M, \rho, \langle\cdot\,,\cdot\rangle,
\llb\cdot\,,\cdot\rrb)$ be a Courant algebroid. That is 
\cite{LiWeXu97,Roytenberg99},
$\rho\colon \mathbb E\to TM$ is a vector bundle morphism over the
identity on $M$,
$\langle\cdot\,,\cdot\rangle\colon \mathbb E\times_M \mathbb E\to
\mathbb R$ is a non-degenerate bilinear pairing and
$\llb\cdot\,,\cdot\rrb$ is an $\R$-bilinear bracket on
$\Gamma(\mathbb E)$ such that
\begin{enumerate}
\item $\llb e_1, \llb e_2, e_3\rrb\rrb= \llb \llb e_1, e_2\rrb, e_3\rrb+ \llb
  e_2, \llb e_1, e_3\rrb\rrb$,
\item $\rho(e_1 )\langle e_2, e_3\rangle= \langle\llb e_1,
  e_2\rrb, e_3\rangle + \langle e_2, \llb e_1 , e_3\rrb\rangle$ and
\item $\llb e_1, e_2\rrb+\llb e_2, e_1\rrb =\rho^t\dr\langle e_1 ,
  e_2\rangle$
\end{enumerate}
for all $e_1, e_2, e_3\in\Gamma(\mathbb E)$ and $f\in C^\infty(M)$.
On the right-hand side of the third equation, $\mathbb E$ is
identified with $\mathbb E^*$ via the pairing.  The identity
$\rho\llb e_1, e_2\rrb = [\rho(e_1),\rho(e_2)]$
follows from the equations above for all
$e_1,e_2\in\Gamma(\mathbb E)$ 
\cite{Uchino02}.

The vector bundle
$\mathbb TM:= TM\oplus T^*M$ over a smooth manifold $M$ together with 
the natural anchor $\rho:=\pr_{TM}$, the symmetric pairing
$\langle (v_p,\theta_p), (w_p,\omega_p)\rangle=\theta_p(w_p)+\omega_p(v_p)$ 
and the \textbf{Courant-Dorfman bracket}
$\llb (X,\theta), (Y, \omega)\rrb=([X,Y],\ldr{X}\omega-\iota_{Y}\dr\theta)$ 
is the prototype of a Courant algebroid.  It is called here the 
\textbf{standard Courant algebroid over $M$}.

\begin{mydef}\label{def_gacs}
	A \textbf{generalised almost complex structure} in $\E$ is a vector 
	bundle morphism $\mathcal J\colon \E\rightarrow \E$ over $\id_M$ 
	such that $\mathcal J^2=-1$ and $\mathcal J$ is orthogonal with 
	respect to the pairing, i.e. $
		\langle \mathcal J(e_1),\mathcal J(e_2)\rangle=
		\langle e_1,e_2\rangle$, 
	for all sections $e_1,e_2\in\Gamma(\E)$.
\end{mydef}

\begin{mydef}\label{def_gcs}
	A generalised almost complex structure $\mathcal J\colon \E\to \E$ is called a
	\textbf{generalised complex structure} in $\E$ if the Nijenhuis 
	tensor of $\mathcal J$ vanishes:
	\[
		0=N_{\mathcal J}(e_1,e_2):=\llbracket e_1,e_2\rrbracket
		-\llbracket \mathcal J(e_1),\mathcal J(e_2)\rrbracket 
		+\mathcal J\bigl(\llbracket \mathcal J(e_1),e_2\rrbracket
		+\llbracket e_1,\mathcal J(e_2)\rrbracket\bigr)\,,
	\]
	for all sections $e_1,e_2\in\Gamma(\E)$.
\end{mydef}

A generalised complex structure in the standard Courant algebroid
$TM\oplus T^*M$ is simply called \textbf{a generalised complex
  structure on $M$}.

\begin{example}\label{ex_gcs_cplx}
	Given an almost complex structure 
	$J\colon TM\rightarrow TM$ the map $\JJ\colon \TT M\rightarrow \TT M$
	\[
		\JJ_J=\left(\begin{matrix}J&0\\
		0& -J^t
		\end{matrix}\right)
	\]
	is a generalised almost complex structure. It is a generalised
        complex structure if and only if $J$ is a complex structure on
        $M$.
\end{example}

Equivalently, generalised complex structures $\JJ$ in $\E$ can be
described by pairs of transversal, complex conjugated Dirac structures
in $\E_\C$, given by the $\pm i$-eigenbundles of $\JJ$. In fact, this
was the original definition in \cite{Hitchin03}, see also
\cite{Gualtieri04,Gualtieri11}.

\subsection{Dorfman connections and the generalised tangent bundle of
  a vector bundle}\label{sec_bg_gentan}
Let $q_E\colon E\to M$ be a vector bundle.  Then the tangent bundle
$TE$ has two vector bundle structures; one as the tangent bundle of
the manifold $E$, and the second as a vector bundle over $TM$. The
structure maps of $TE\to TM$ are the derivatives of the structure maps
of $E\to M$. The space $TE$ is a double vector bundle with core bundle
$E \to M$.
Linear splittings of $TE$ are equivalent to
linear horizontal subspaces $H\subseteq TE$, which in turn are
equivalent to linear $TM$-connections $\nabla$ in $E$.  For details on
these double vector bundles, their core and linear sections, on linear
splittings and on connections, consult \cite{Mackenzie05,Jotz18a}.

Dualising $TE$ as a vector bundle over $E$ gives the cotangent
bundle $T^*E$, which is itself a double vector bundle with sides $E$ and $E^*$
and core $T^*M$, see \cite{Mackenzie05}.
\begin{equation*}
\begin{xy}
\xymatrix{
TE \ar[d]_{Tq_E}\ar[r]^{p_E}& E\ar[d]^{q_E}\\
 TM\ar[r]_{p_M}& M}
\end{xy}\,\qquad
\begin{xy}
\xymatrix{
T^*E\ar[r]^{c_E}\ar[d]_{r_E} &E\ar[d]^{q_E}\\
E^*\ar[r]_{q_{E^*}}&M
}
\end{xy}
\end{equation*}
Consider the direct sum over $E$ of these two double vector bundles,
\begin{align*}
\begin{xy}
\xymatrix{
TE\oplus T^*E\ar[r]^{\quad \pi_E}\ar[d]_{\Phi_E} &E\ar[d]^{q_E}\\
TM\oplus E^*\ar[r]_{\quad q_{TM\oplus E^*}}&M
}\end{xy}
\end{align*}
with $\Phi_E=Tq_E\oplus r_E$.  
A subbundle $L\subseteq TE\oplus T^*E$ that is closed under the
addition of $TE\oplus T^*E\to TM\oplus E^*$, and complementary to
$T^qE\oplus (T^qE)^\circ$, is called a \textbf{linear horizontal subspace in
$TE\oplus T^*E$}. 

In the following, for any section $(e,\theta)$ of $E\oplus T^*M$, the
vertical section $(e,\theta)^\uparrow\in\Gamma_E(T^{q_E}E\oplus
(T^{q_E}E)^\circ)$ is the pair defined by
\[
(e,\theta)^\uparrow(e_m')=\left(\left.\frac{d}{dt}\right\an{t=0}e_m'+te(m), (T_{e'_m}q_E)^t\theta(m)\right)
\]
for all $e_m'\in E$.  By construction this is a core section of
$TE\oplus T^*E\to E$.  For any section $\phi$
of $\Hom(E,E\oplus T^*M)$, the core-linear section
$\widetilde\phi\in \Gamma_E^l(T^{q_E}E\oplus (T^{q_E}E)^\circ$ is
defined by $\widetilde\phi(e(m))=(\phi(e))^\uparrow(e(m))$ for all
$e\in\Gamma(E)$.  The double vector bundle $TE\oplus T^*E$ is
described in more detail in \cite{Jotz18a}, where also an equivalence
of linear splittings of $TE\oplus T^*E$ with Dorfman connections is
established.

A \textbf{Dorfman
  $TM\oplus E^*$-connection on $E\oplus T^*M$} is an $\R$-bilinear map
\[\Delta\colon \Gamma(TM\oplus E^*)\times\Gamma(E\oplus T^*M)\to\Gamma(E\oplus T^*M)\]
satisfying \cite{Jotz18a} \begin{enumerate}
\item $\Delta_\nu(f\cdot\tau)=f\cdot\Delta_\nu\tau+\ldr{\pr_{TM}(\nu)}(f)\cdot \tau$,
\item $\Delta_{f\cdot \nu}\tau=f\cdot \Delta_\nu\tau+\langle \nu,\tau\rangle\cdot(0,\dr f)$, and 
\item $\Delta_\nu(0,\dr f)=(0,\dr(\ldr{\pr_{TM}\nu}f))$
\end{enumerate}
for all $\nu\in\Gamma(TM\oplus E^*)$, $\tau\in\Gamma(E\oplus T^*M)$ and
$f\in C^\infty(M)$. By the first axiom, $\Delta$ defines a map
$\Delta\colon \nu\mapsto \Delta_\nu\in\operatorname{Der}(E\oplus
T^*M)$. The dual of this map in the sense of derivations defines a
\textbf{dull bracket on sections of $TM\oplus E^*$}, i.e.~an
$\R$-bilinear map
\[\llb\cdot\,,\cdot\rrb_\Delta\colon \Gamma(TM\oplus E^*)\times
  \Gamma(TM\oplus E^*)\to \Gamma(TM\oplus E^*)\] satisfying
\begin{enumerate}
\item $\pr_{TM}\llb
  \nu_1,\nu_2\rrb_\Delta=[\pr_{TM}\nu_1,\pr_{TM}\nu_2]$,
\item
  $\llb f_1\nu_1, f_2\nu_2\rrb=f_1f_2\llb
  \nu_1,\nu_2\rrb_\Delta+f_1\ldr{\pr_{TM}\nu_1}(f_2)\nu_2-f_2\ldr{\pr_{TM}\nu_2}(f_1)\nu_1$
\end{enumerate}
for all $\nu_1,\nu_2\in\Gamma(TM\oplus E^*)$ and $f_1,f_2\in
C^\infty(M)$.

Since the vector bundle $TM\oplus E^*$ is anchored by the morphism
$\pr_{TM}\colon TM\oplus E^*\to TM$, the $TM$-part of
$\llb \nu_1, \nu_2\rrb_\Delta +\llb \nu_2, \nu_1\rrb_\Delta$ is trivial and
this sum can be seen as an element of $\Gamma(E^*)$. Let
$\Jac_\Delta\in\Omega^3(TM\oplus E^*,TM\oplus E^*)$ be the Jacobiator
of the dull bracket
$\DMbra{\cdot}{\cdot}_{\Delta}$. Then
\begin{equation*}
\begin{split}
\Jac_\Delta(\nu_1,\nu_2,\nu_3):=\DMbra{\DMbra{\nu_1}{\nu_2}_\Delta}{\nu_3}_\Delta
	+\text{cyclic permutations}=
	R_\Delta(\nu_1,\nu_2)^t(\nu_3),
      \end{split}
    \end{equation*}
    with $R_\Delta\in\Omega^2(TM\oplus E^*, \Hom(E\oplus T^*M, E))$
    the curvature of the Dorfman connection. Hence a skew-symmetric
    dull bracket is a Lie algebroid bracket if and only if the
    corresponding Dorfman connection is flat.

Linear splittings of $TE\oplus T^*E$ are in bijection with dull
brackets on sections of $TM\oplus E^*$, or equivalently with Dorfman
connections
$\Delta\colon \Gamma(TM\oplus E^*)\times\Gamma(E\oplus
T^*M)\to\Gamma(E\oplus T^*M)$, see \cite{Jotz18a}. Choose such a
Dorfman connection. The horizontal lift
$\sigma:=\sigma_{TM\oplus E^*}^\Delta\colon\Gamma(TM\oplus
E^*)\to\Gamma^l_E(TE\oplus T^*E)$ is given by
\begin{equation}\label{eq_slift}
  \sigma(X,\epsilon)(e(m))=\left(T_me X(m), \dr
    \ell_\epsilon(e(m))\right)-\Delta_{(X,\epsilon)}(e,0)^\uparrow(e(m))
\end{equation}
for $e\in\Gamma(E)$ and any pair
$(X,\epsilon)\in\Gamma(TM\oplus E^*)$.  The natural pairing on fibres of
$TE\oplus T^*E\to E$ is then given by \cite{Jotz18a}
 $\left\langle \sigma(\nu_1), \sigma(\nu_2)\right\rangle=\ell_{\llb \nu_1, \nu_2\rrb_\Delta +\llb \nu_2, \nu_1\rrb_\Delta}$,
 $\left\langle \sigma(\nu),
    \tau^\uparrow\right\rangle=q_E^*\langle \nu, \tau\rangle$, and 
 $\left\langle \tau_1^\uparrow,
    \tau_2^\uparrow\right\rangle=0$
for $\nu, \nu_1,\nu_2\in\Gamma(TM\oplus E^*)$ and
$\tau,\tau_1, \tau_2\in\Gamma(E\oplus T^*M)$.  The following equations
follow for $\varphi,\psi\in\Gamma(\Hom(E,E\oplus T^*M))$,
$\nu\in\Gamma(TM\oplus E^*)$ and $\tau\in\Gamma(E\oplus T^*M)$
\begin{equation}\label{pairing_core_lin}
  \begin{split}
    \pair{\corelinear{\varphi}}{\slift(\nu)} =\ell_{\varphi^*(\nu)},
    \quad \langle\corelinear{\varphi},\tau^\uparrow\rangle= 0,\quad
    \langle \corelinear{\varphi},\corelinear{\psi}\rangle=0.
                      \end{split}
                    \end{equation}

                    The Courant-Dorfman bracket on sections of
                    $TE\oplus T^*E\to E$ is given by \cite{Jotz18a}
 \begin{enumerate}
 \item
   $\left\llb \sigma(\nu_1), \sigma(\nu_2)\right\rrb=\sigma\llb \nu_1,
   \nu_2\rrb_\Delta- \widetilde{R_\Delta(\nu_1,\nu_2)\circ \iota_E}$,
 \item
   $\left\llb \sigma(\nu), \tau^\uparrow\right\rrb=(\Delta_\nu\tau)^\uparrow$, and 
\item $\left\llb \tau_1^\uparrow,
    \tau_2^\uparrow\right\rrb=0$
\end{enumerate}
for $\nu, \nu_1,\nu_2\in\Gamma(TM\oplus E^*)$ and
$\tau,\tau_1, \tau_2\in\Gamma(E\oplus T^*M)$. Here,
$\iota_E\colon E\to E\oplus T^*M$ is the canonical inclusion.

The anchor $\Theta=\pr_{TE}\colon TE\oplus T^*E\to TE$ restricts to
the map $\partial_{E}=\pr_{E}\colon E\oplus T^*M\to E$ on the cores,
and defines an anchor
$\rho_{TM\oplus E^*}=\pr_{TM}\colon TM\oplus E^*\to TM$ on the side.
More precisely, the anchor of $(e,\theta)^\uparrow$ is
$e^\uparrow\in \mx^c(E)$ and
$\Theta(\sigma(\nu))=\widehat{\nabla_\nu^*}\in\mx(E)$, where the
linear connection
$\nabla\colon \Gamma(TM\oplus E^*)\times\Gamma(E)\to \Gamma(E)$ is
defined by $\nabla_\nu=\pr_E\circ\Delta_\nu\circ\iota_E$ for all
$\nu\in\Gamma(TM\oplus E^*)$.

      \begin{example}\label{def_std_DMC}
        Let $q\colon E\to M$ be a smooth vector bundle. Since a linear
        connection $\nabla\colon \mx(M)\times\Gamma(E)\to\Gamma(E)$ is
        equivalent to a linear horizontal space
        $H_\nabla\subseteq TE$, it also defines a linear horizontal
        space $H_\nabla\oplus H_\nabla^\circ\subseteq TE\oplus T^*E$. The corresponding Dorfman connection
        \[\Delta\colon\Gamma(TM\oplus E^*)\times\Gamma(E\oplus T^*M)\to\Gamma(E\oplus T^*M)
        \]
        is given by
        $\Delta_{(X,\epsilon)}(e,\theta)=(\nabla_Xe, \ldr{X}\theta+\langle\nabla^*\epsilon, e\rangle)
        $
        for $X\in\mx(M)$, $\theta\in\Omega^1(M)$, $e\in\Gamma(E)$ and
        $\epsilon\in\Gamma(E^*)$. This is the \textbf{standard Dorfman
          connection} defined by $\nabla$.
The corresponding dull bracket is   \begin{equation}\label{eq_std_bra}
\DMbra{(X,\epsilon)}{(Y,\eta)}_\Delta=
\Bigl([X,Y],\nabla^*_X\eta-\nabla^*_Y\epsilon\Bigr)\,
\end{equation}
for $X,Y\in\mx(M)$ and $\epsilon, \eta\in\Gamma(E^*)$.
        \end{example}

\medskip

The remainder of this section discusses changes of linear splittings
of $TE\oplus T^*E$.
\begin{mydef}\label{def_change_DMC}
	Given two $(TM\oplus E^*)$-Dorfman connections $\Delta^1$ 
	and $\Delta^2$ on $E\oplus T^*M$ with corresponding lifts $\sigma_1$
	and $\sigma_2$, the \textbf{change of splitting} from $\Delta^1$ to
	$\Delta^2$ is $\Phi_{12}\in\Gamma\bigl((TM\oplus E^*)^*\otimes 
	\Hom(E,E\oplus T^*M)\bigr)$ defined by the equation
	\[
		\corelinear{\Phi_{12}(\nu)}:=\sigma_2(\nu)-\sigma_1(\nu)\,,
	\]
	for any $\nu\in\Gamma(TM\oplus E^*)$. The change of splitting is called 
	\textbf{skew-symmetric} if 
	$\Psi_{12}(\nu_1,\nu_2):=\Phi_{12}(\nu_1)^t(\nu_2)$
	is skew-symmetric, that is $\Psi_{12}\in\Omega^2(TM\oplus E^*,E^*)$. 
	The form $\Psi_{12}$ is also called \textbf{change of splittings}. 
\end{mydef}
\begin{lemma}
	Given two Dorfman connections $\Delta^1$, $\Delta^2$ as above, 
	their corresponding dull brackets are related by 
	\begin{equation}\label{eq_change_bra}
		\DMbra{\nu_1}{\nu_2}_{\Delta^2}
		=\DMbra{\nu_1}{\nu_2}_{\Delta^1}
		+\bigl(0,\Psi_{12}(\nu_1,\nu_2)\bigr)\,.
	\end{equation} 
\end{lemma}
\begin{proof}
	The definition of the change of splittings together with the 
	correspondence between lifts and Dorfman connections as in
	 \eqref{eq_slift}
	immediately gives that for any $\nu\in\Gamma(TM\oplus E^*)$ and 
	$\tau \in\Gamma(E\oplus T^*M)$
	\[
		\Delta^2_\nu\tau
		=\Delta^1_\nu\tau
		-\Phi_{12}(\nu)\bigl(\pr_E\tau\bigr)\,.
	\]
	Again dualising this equation gives the desired formula 
	\eqref{eq_change_bra} for the change of splittings for the 
	corresponding dull brackets.
\end{proof}
An immediate consequence is the following corollary. 
\begin{cor}
	If the Dorfman connection $\Delta_1$ is skew-symmetric, then $\Delta_2$ 
	is skew-symmetric if and only if the change of splitting is skew-symmetric.
      \end{cor}

      \subsection{VB-Courant algebroids}
      The linear Courant algebroid on $TE\oplus T^*E$ is a prototype
      of a VB-Courant algebroid. This section gives the general
      definition of a VB-Courant algebroid.

      A \textbf{metric double
        vector bundle} \cite{Jotz18a} is a double vector bundle
      $(D;A,B;M)$ equipped with a symmetric, non-degenerate fibrewise
      pairing $D\times_B D\rightarrow \R$, such that the induced map
      $D\rightarrow D^*_B$ is an isomorphism of double vector bundles.
      In particular the core must be isomorphic to $A^*$.
      A \textbf{VB-Courant algebroid} $(\E;Q,B;M)$
	is a metric double vector bundle 
	\[
		\begin{tikzcd}
		\E \ar[r] \ar[d] & B \ar[d]\\
		Q \ar[r] & M
		\end{tikzcd} 
	\]
	such that $\E\rightarrow B$ is a Courant algebroid,
	the anchor $\rho_\E\colon \E\rightarrow TB$ is linear, i.e.~a morphism of 
	double vector bundles over some morphism $\rho_Q\colon Q\rightarrow TM$ 
	and the Courant bracket is linear, that is
	\[\DMbra{\Gamma^\ell_B(\E)}{\Gamma^\ell_B(\E)}
		\subseteq\Gamma^\ell_B(\E), \quad \DMbra{\Gamma^\ell_B(\E)}{\Gamma^c_B(\E)}
		\subseteq\Gamma^c_B(\E), \quad \text{ and }\quad \DMbra{\Gamma^c_B(\E)}{\Gamma^c_B(\E)}=0.\]

	Given a VB-Courant algebroid $(\E;Q,B;M)$, a \textbf{VB-Dirac structure}
	in $\E$ is a sub-double vector bundle $(D;U,B;M)$ with $U\subseteq Q$ 
	such that $D\to B$ is a Dirac structure in $\E\to B$.

\begin{example}\label{ex_std_VB_CA}
  The standard Courant algebroid $\E=TE\oplus T^*E$ over a vector
  bundle $E$ is a VB-Courant algebroid with $Q=TM\oplus E^*$ and
  $B=E$. The subspaces $TE$ and $T^*E$ are VB-Dirac structures in
  $\E$.
\end{example}
\begin{example}\label{ex_tan_double_VB_CA}
  The tangent double of a Courant algebroid $E\rightarrow M$ is a
  VB-Courant algebroid, where $\E=TE$, $Q=E$ and $B=TM$. The anchor of
  $TE$ is given by $I\circ T\rho_E\colon TE\to T(TM)$, where
  $I\colon TTM\to TTM$ is the canonical involution \cite{Tulczyjew89,Mackenzie05},
  exchanging the two
  vector bundle structures $Tp_M$ and $p_{TM}$ of $TTM\to TM$.
\end{example}

\subsection{The generalised tangent bundle of a Lie algebroid}
Let here $A\to M$ be a Lie algebroid.
Then for $a\in\Gamma(A)$, the derivations $\LL_a$ of $\Gamma(TM\oplus A^*)$ 
and of $\Gamma(A\oplus T^*M)$ over $\rho(a)$ are defined by
	\begin{align*}
          \LL_a(X,\alpha)&:=\bigl([\rho(a),X],\LL_a\alpha\bigr),
                           \qquad \LL_a(a',\theta):=\bigl([a,a'],\LL_{\rho(a)}\theta\bigr)\,
	\end{align*}
        for $(X,\alpha)\in\Gamma(TM\oplus A^*)$ and
        $(a',\theta)\in\Gamma(A\oplus T^*M)$.
Fix a skew-symmetric Dorfman
connection
$	\Delta\colon\Gamma(TM\oplus A^*)\times \Gamma(A\oplus T^*M)
	\rightarrow \Gamma(A\oplus T^*M)$
and set \cite{Jotz18a}
	\begin{equation*}
		\Omega\colon \Gamma(TM\oplus A^*)\times \Gamma(A)\to 
		\Gamma(A\oplus T^*M), \quad \Omega_{(X,\alpha)}a:=\Delta_{(X,\alpha)}(a,0)-(0,\dr\langle\alpha, a\rangle).
	\end{equation*}
	Define then the \textbf{basic connections associated to $\Delta$}  by
	\begin{align*}
	\nabla^\bas_a(X,\alpha)&
	:=(\rho,\rho^t)\bigl(\Omega_{(X,\alpha)}a\bigr)+\LL_a(X,\alpha)\,,\\
	\nabla^\bas_a(a',\theta)&
	:=\Omega_{(\rho,\rho^t)(a',\theta)}a+\LL_a(a',\theta),
	\end{align*}
        for $a\in\Gamma(A)$, $(X,\alpha)\in\Gamma(TM\oplus A^*)$ and
        $(a',\theta)\in\Gamma(A\oplus T^*M)$.  These are ordinary
        linear $A$-connections on $TM\oplus A^*$ and $A\oplus T^*M$,
        respectively, which are dual to each other \cite{Jotz18a}.

The \textbf{basic 
	curvature} $R_\Delta^\bas\in\Omega^2(A,\Hom(TM\oplus A^*,A\oplus T^*M))$ 
	associated to $\Delta$ is defined by 
	\begin{equation}\label{eq_bas_curv}
	R^\bas_\Delta(a_1,a_2)\nu
	:=-\Omega_\nu [a_1,a_2]
	+\LL_{a_1}\bigl(\Omega_\nu a_2\bigr)
	-\LL_{a_2}\bigl(\Omega_\nu a_1\bigr)
        +\Omega_{\nabla^\bas_{a_2}\nu}a_1
	-\Omega_{\nabla^\bas_{a_1}\nu}a_2
      \end{equation}
      for $a_1,a_2\in\Gamma(A)$ and 
	$\nu\in\Gamma(TM\oplus A^*)$.
\cite{Jotz18a} shows 
	that $R_{\nabla^\bas}=R^\bas_\Delta\circ (\rho,\rho^t)$
	and $R_{\nabla^\bas}=(\rho,\rho^t)\circ R^\bas_\Delta$. 

\begin{thm}\label{thm_TTA_LA}\cite{Jotz18a}
  Let $A\to M$ be a Lie algebroid with anchor $\rho$ and let
  $\Delta$ a skew-symmetric Dorfman $TM\oplus A^*$-connection on
  $A\oplus T^*M$. Write $\Theta$ for the anchor of the Lie algebroid
  $\TT A$. Then 
	\begin{enumerate}
		\item $[\sigma^\Delta_A(a_1),\sigma^\Delta_A(a_2)]
			=\sigma^\Delta_A\bigl([a_1,a_2]\bigr)
			-\corelinear{R^\bas_\Delta(a_1,a_2)}$\,,
		\item $[\sigma^\Delta_A(a),\tau^\dagger]
			=\bigl(\nabla^\bas_a\tau\bigr)^\dagger$\,,
		\item $[\tau_1^\dagger,\tau_2^\dagger]=0$\,,
		\item $\Theta\bigl(\sigma^\Delta_A(a)\bigr)
		=\widehat{\nabla^\bas_a}\in\XX^\ell(TM\oplus A^*)$\,,
		\item $\Theta\bigl(\tau^\dagger\bigr)
		=\bigl((\rho,\rho^t)\tau\bigr)^\uparrow\in\XX^c(TM\oplus A^*)$
              \end{enumerate}
              for $a,a_1,a_2\in\Gamma(A)$ and
  $\tau,\tau_1,\tau_2\in\Gamma(A\oplus T^*M)$.
\end{thm}
That is \cite{Jotz18a}, the complex
$(\rho,\rho^t)\colon (A\oplus T^*M)[0]\to (TM\oplus A^*)[1]$, the
basic connections $\nabla^\bas$ and the basic curvature
$R^\bas_\Delta$ define the 2-representation corresponding to the
VB-algebroid $(TA\oplus_A T^*A\to TM\oplus A^*, A\to M)$
\cite{Jotz18a} in the decomposition corresponding to $\Delta$, see
\cite{GrMe10}.

\section{Holomorphic vector bundles and holomorphic Lie algebroids}\label{new_complex_section}
\label{sec_hol_vb}
This section gives a direct proof that a holomorphic structure on a 
vector bundle is equivalent to a linear complex structure on it. In 
particular, ``adapted'' connections are described in the language of 
linear complex structures. Finally, a holomorphic Lie algebroid $A$ 
with a choice of adapted connection gives rise to infinitesimal ideal 
systems in $A_\C$.

\subsection{Linear almost complex structures via connections}\label{new_complex_subsection}
Let $E$ be a vector bundle over a manifold $M$. A complex structure
in the fibres of $E$ is equivalent to a vector bundle morphism 
$j_E\colon E\to E$ over $\id_M$, such that $j_E^2=-\id_E$. 
Consider 
any connection $\nabla\colon \mx(M)\times\Gamma(E)\to\Gamma(E)$.  
Then the connection
\[\tilde\nabla\colon\mx(M)\times\Gamma(E)\to\Gamma(E),\qquad
  \tilde\nabla_Xe=\half(\nabla_Xe-j_E(\nabla_X(j_E(e))))
 \]
 satisfies $\tilde\nabla j_E=0$. Such a connection is called 
 \emph{complex-linear} connection on $E$, since it is 
 $\mathbb C$-linear in its second argument.

      Consider a holomorphic vector bundle $E\to M$.  Since $E$ is
      locally generated as a $\mathbb C$-vector bundle by holomorphic
      sections $e_1,\ldots, e_k$ of $q_E\colon E \to M$, the real
      vector bundle $E$ is locally generated by the holomorphic
      sections
      $e_1,\ldots,e_k, f_1:=i\cdot e_1, \ldots, f_k:=i\cdot e_k$.
      Then since these sections are holomorphic, they satisfy
  \begin{equation}\label{hol_sections_J}
    J_E\circ Te_l=Te_l\circ J_M \qquad \text{ and } \qquad J_E\circ Tf_l=Tf_l\circ J_M
  \end{equation}
  for $l=1,\ldots,k$. 
  The following lemma is easy to see in  holomorphic frames.
  \begin{lemma}\label{lem_int_1}
    Let $q_E\colon E\to M$ be a holomorphic vector bundle.
    If $e\in\mathcal E(U)$ is a holomorphic
  section, then the vector field $e^\uparrow \in\mx(E)$ is holomorphic as
  well, and for any $e\in\Gamma(E)$, the complex structure $J_E$ sends
  $e^\uparrow$ to $(j_E e)^\uparrow$.
\end{lemma}
That is, the complex structure $J_E\colon TE\to TE$ is a double vector
bundle morphism over the identity on $E$, the complex structure
$J_M\colon TM\to TM$ of the base $M$, and with core $j_E\colon E\to E$
-- as in \eqref{diag_JE}.
  The goal of this section is the proof of the converse: if a linear
  almost complex structure as in \eqref{diag_JE} is integrable, then
  the induced complex structures on $E$ and on $M$ make $E\to M$ a
  holomorphic vector bundle.  Consider therefore for the remainder of
  this section  a smooth
  vector bundle $E\to M$ with such a linear almost complex structure $(J_E,J_M)$
  as in \eqref{diag_JE} 
  with core morphism $j_E\colon E\to E$.
 \begin{mydef}\label{definition_adapted}
   A linear connection
   $\nabla\colon \mx(M)\times\Gamma(E)\to \Gamma(E)$ is
   \textbf{adapted to $J_E$} if the corresponding linear splitting
   $\Sigma\colon TM\times_ME\to TE$ of $TE$ satisfies
   $J_E\Sigma(v,e)=\Sigma(J_Mv,e)$ for all $(v,e)\in TM\times_M E$.
   Equivalently, the corresponding horizontal lift 
   $\sigma^\nabla\colon \mx(M)\to\mx^\ell(E)$ lifts the almost complex 
   structure on $M$ to the one on $E$, that is 
   $\sigma^\nabla(J_MX)=J_E\sigma^\nabla(X)$ for all $X\in\mx(M)$. 
    \end{mydef}
        
        Consider any linear splitting $\Sigma$ of the double vector bundle $TE$
  corresponding to a linear connection $\nabla\colon\mx(M)\times\Gamma(E)\to\Gamma(E)$.
  Then for all $X\in\mx(M)$, the vector field $J_E(\widehat{\nabla_X})$ is
  $q$-related with $J_M(X)\in \mx(M)$.  Hence
  \begin{equation*}
    J_E(\widehat{\nabla_X})=\widehat{\nabla_{J_MX}}+\widetilde{\psi(X)}
  \end{equation*}
  for a section $\psi(X)$ of $\End(E)$, which defines a form 
$\psi\in\Omega^1(M,\operatorname{End}(E))$. 
Since $J_E^2=-1$ and $J_M^2=-1$, this form satisfies
\[
  \psi(J_MX)=-j_E\circ\psi(X)
\]
for all $X\in\mx(M)$.
  Consider
  $\Sigma'\colon E\times_MTM\to TE$,
  \[ \Sigma'(e_m,v_m)=\Sigma(e_m,v_m)+_{TM}
  \bigl(T_m0^Ev_m+_{E}\half\overline{j_E\psi(v_m)(e_m)}\bigr)\,.
  \]
	Then a simple computation shows that $\Sigma'$ satisfies the condition
	of Definition \ref{definition_adapted}, hence showing the following
	proposition. 
  \begin{prop}\label{existence_adapted}
    Let $E\to M$ be a smooth vector bundle endowed with a linear almost
    complex structure $J_E\colon TE\to TE$ over $J_M\colon TM\to TM$,
    with core morphism $j_E\colon E\to E$.  
  Then there exists a linear $TM$-connection on $E$ that is adapted to $J_E$.
\end{prop}

By definition, a linear connection
$\nabla\colon \mx(M)\times\Gamma(E)\to\Gamma(E)$ is adapted to $J_E$
if and only if the following identity holds for all $e\in\Gamma(E)$,
$X\in\mx(M)$:
\begin{equation*}
  \begin{split}
    J_E(T_meX_m)-\left.\frac{d}{dt}\right\an{t=0}e_m+tj_E(\nabla_Xe)(m)
    &=J_E(\widehat{\nabla_X}(e(m)) =\widehat{\nabla_{J_MX}}(e(m))
    \\
  &=T_meJ_MX_m-\left.\frac{d}{dt}\right\an{t=0}e_m+t(\nabla_{J_MX}e)(m).
  \end{split}
  \end{equation*}
  
    In particular, if $E\to M$ is a holomorphic vector bundle, and 
  $e$ a holomorphic section, then $J_E\circ Te=Te\circ J_M$. Hence,
    $\nabla$ is adapted to $J_E$ if and only if  
    \begin{equation}\label{adapted_holo}
      j_E\nabla_Xe=\nabla_{J_MX}e
    \end{equation}
    for all $X\in\mx(M)$ and all \emph{holomorphic sections} $e\in\Gamma(E)$.
   Given in that case any $\C$-linear connection $\tilde{\nabla}$, then the
    connection $\nabla$, defined by 
    \begin{equation}
      \label{nabla_adapted}
      \nabla_Xe=\half\left(\tilde{\nabla}_Xe-j_E\tilde{\nabla}_{J_MX}e\right)
    \end{equation}
    for $X\in\mx(M)$ and \emph{holomorphic sections} $e\in\Gamma(E)$,
    is adapted to $J_E$ and $\C$-linear. 

      In general the following theorem shows that the
      existence of such a $\C$-linear adapted connection follows
      easily from the integrability of the linear almost complex
      structure $J_E$.

\begin{thm}\label{thm_dec_lin_com}
  Let $E\to M$ be a smooth vector bundle endowed with a linear almost
  complex structure $J_E\colon TE\to TE$ over $J_M\colon TM\to TM$,
  with core morphism $j_E\colon E\to E$.  Then $J_E$ is integrable if
  and only if
  \begin{enumerate}
    \item $J_M$ is integrable and
    \item there exists a $\mathbb C$-linear connection
  $\nabla\colon \mx(M)\times\Gamma(E)\to\Gamma(E)$ that is adapted to
  $J_E$,
\item the form $N_{R_\nabla,j_E}\in\Omega^2(M,\End(E))$  defined
    by
    \[N_{R_\nabla,j_E}(X,Y):=R_\nabla(X,Y)-R_\nabla(J_MX,J_MY)+j_E R_\nabla(J_MX,Y)
    +j_E R_\nabla(X,J_MY)\]
  vanishes.
  \end{enumerate}
\end{thm}

\begin{rmk}
       In the setting above, a linear connection
        $\nabla\colon\mx(M)\times\Gamma(E)\to\Gamma(E)$ is adapted to $J_E$ 
        if and only if 
      $J_E(H_\nabla)=H_\nabla$.
      It satisfies $\nabla j_E=0$ if and only if
      $H_\nabla=Tj_E(H_\nabla)$.
    \end{rmk}
    
\begin{proof}[Proof of Theorem \ref{thm_dec_lin_com}]
  Choose a linear connection
  $\nabla\colon \mx(M)\times\Gamma(E)\to\Gamma(E)$ that is adapted to
  $J_E$ as in Definition \ref{definition_adapted} -- the existence of
  such a connection is given by Proposition \ref{existence_adapted}.
  That is, the linear vector fields $\widehat{\nabla_X}$, for
  $X\in\mx(M)$, all satisfy
  \[ J_E\circ \widehat{\nabla_X}=\widehat{\nabla_{J_MX}}.
  \]
  Then
  \begin{equation}\label{nij1}
    \begin{split}
      N_{J_E}\left(\widehat{\nabla_X},
        \widehat{\nabla_Y}\right)&=\left[\widehat{\nabla_X},\widehat{\nabla_Y}\right]-\left[J_E\widehat{\nabla_X},J_E\widehat{\nabla_Y}\right]
      +J_E\left[J_E\widehat{\nabla_X},\widehat{\nabla_Y}\right]+J_E\left[\widehat{\nabla_X},J_E\widehat{\nabla_Y}\right]
      \\
      &=
      \widehat{\nabla_{N_{J_M}(X,Y)}}-\widetilde{N_{R_\nabla,j_E}(X,Y)},
      \end{split}
    \end{equation}
  \begin{equation*}
    \begin{split}
    N_{J_E}\left(\widehat{\nabla_X}, e^\uparrow\right)&=\left[\widehat{\nabla_X},e^\uparrow\right]
		-\left[\widehat{\nabla_{J_MX}}, (j_Ee)^\uparrow\right]
		+(j_E\nabla_{J_MX}e)^\uparrow
		+(j_E\nabla_Xj_Ee)^\uparrow\\
                &=\left(\nabla_Xe-\nabla_{J_MX}j_Ee+j_E\nabla_{J_MX}e+j_E\nabla_Xj_Ee\right)^\uparrow.
                \end{split}
              \end{equation*}
              and
              \begin{equation*}
                N_{J_E}\left(e_1^\uparrow, e_2^\uparrow\right)=0
              \end{equation*}
              for all $X,Y\in\mx(M)$ and $e,e_1,e_2\in\Gamma(E)$.
              Since $\mx(E)$ is spanned as a $C^\infty(E)$-module by
              these linear and core sections, $N_{J_E}$ vanishes if
              and only if $N_{J_M}(X,Y)=0$ (by projecting \eqref{nij1}
              to $M$), $N_{R_\nabla,j_E}(X,Y)=0$ and
              \begin{equation}\label{nij2}
                \nabla_Xe-\nabla_{J_MX}j_Ee+j_E\nabla_{J_MX}e+j_E\nabla_Xj_Ee=0
              \end{equation}
              for all $X,Y\in\mx(M)$ and $e\in\Gamma(E)$. Therefore,
              if $N_{J_M}=0$, $N_{R_\nabla,j}=0$ and $\nabla j_E=0$, the almost
              complex structure $J_E$ is integrable.

              Assume conversely that $N_{J_E}=0$. Then $N_{J_M}=0$ and $J_M$ is integrable.  
             Define now a $\C$-linear connection $\tilde{\nabla}$ 
             by
            \[ \tilde{\nabla}_Xe=\half\nabla_Xe-\half j_E\nabla_X j_Ee\,.
            \]
            The difference between $\nabla$ and $\tilde{\nabla}$ is the form 
            $\omega\in\Omega^1(M,\End(E))$, given for $X\in\mx(M)$ and $e\in\Gamma(E)$ by 
            \[\omega(X)(e)=\half(\nabla_Xe+j_E\nabla_Xj_Ee)\,.
            \]
            By \eqref{nij2} $j_E\circ (\omega(X))=\omega(J_MX)$ and hence                        
            \begin{equation*}
              J_E\left(\widehat{\tilde{\nabla}_X}\right)=J_E\left(\widehat{\nabla_X}-
              \widetilde{\omega(X)}\right)
              =\widehat{\nabla_{J_MX}}-\widetilde{j_E\omega(X)}
              =\widehat{\nabla_{J_MX}}-\widetilde{\omega(J_MX)}
              =\widehat{\tilde{\nabla}_{J_MX}},
            \end{equation*}
            so $\tilde{\nabla}$ is still adapted to $J_E$. Therefore, by
            \eqref{nij1}, $N_{R_{\tilde{\nabla}},j_E}=0$.
  \end{proof}
 
  In the situation of the previous theorem, consider a complex-linear
  connection $\nabla\colon \mx(M)\times\Gamma(E)\to\Gamma(E)$, that is
  adapted to $J_E$. Consider the complexification
  $D\colon \Gamma(T_{\C}M)\times\Gamma(E)\to\Gamma(E)$ of $\nabla$ in 
  the first argument. Then $D$ is still $\C$-linear and   
its curvature
$R_D\in\Omega^2(T_\C M,\End_\C(E))$ is the complexification of
$R_\nabla$ in the first two arguments. 
The straight-forward
proof of the following proposition is left to the reader
\begin{prop}\label{prop_adapted_chern}
  In the situation of Theorem \ref{thm_dec_lin_com}, with $J_M$ 
  integrable and a $\C$-linear, adapted $\nabla$, then $N_{R_\nabla,j_E}=0$ 
  if and only if the complexification
  $D\colon \Gamma(T_{\mathbb C}M)\times\Gamma(E)\to\Gamma(E)$ splits into
      \begin{equation}\label{chern}
        D=D^{1,0}+D^{0,1}
        \end{equation}
      with $D^{1,0}\colon \Gamma(T^{1,0}M)\times\Gamma(E)\to\Gamma(E)$
     a linear connection and
     $D^{0,1}\colon\Gamma(T^{0,1}M)\times\Gamma(E)\to\Gamma(E)$
     a flat connection.
    \end{prop}

\begin{rmk}\label{complexification_rem}
  Let $j_\C\colon E_\C\to E_\C$ be the complexification of $j_E\colon E\to E$,
  and let $E^{1,0}$ and $E^{0,1}$ be as usual the eigenspaces of $j_\C$
  to the eigenvalues $i$ and $-i$, respectively.
 The complexification of $D$ in the
  second argument, or of $\nabla$ in both arguments, written
  $\nabla^\C\colon \Gamma(T_{\mathbb C}M)\times\Gamma(E_{\mathbb
    C})\to \Gamma(E_{\mathbb C})$ clearly preserves $j_\C$.
As a consequence, $\nabla^\C$ preserves $E^{1,0}$ and $E^{0,1}\subseteq E_\C$.
Denote by $\theta\colon E\to E^{1,0}$ the canonical isomorphism of 
$\C$-vector bundles, given by $\theta(e)=\half(e-ij_E(e))$.
Then the equality \[\nabla^\C_X\theta(e)= \theta(D_Xe)\] is immediate
for all $e\in\Gamma(E)$ and $X\in\Gamma(T_{\mathbb C}M)$. That is,
modulo the isomorphism $\theta$, the restriction of $\nabla^\C$ to
$E^{1,0}$ coincides with $D$.
  \end{rmk}
Together with the standard correspondence between holomorphic
structures on a vector bundle and flat connections
$D^{0,1}\colon \Gamma(T^{0,1}M)\times\Gamma(E)\to\Gamma(E)$ (see
\cite{Rawnsley79} and \cite{AtHiSi78}), 
this shows that an integrable linear
complex structure on a smooth vector bundle gives rise to a
holomorphic structure on $E$. To obtain the desired one-to-one
correspondence, it only remains to show that the map $J_E$ is
indeed the complex structure of the complex manifold $E$.
In order to do that, it is sufficient to prove that
    if $e\in\Gamma(E)$ is $D^{0,1}$-flat, i.e.~if $e$ is a holomorphic
    section on $E$, then
    \[ \left[\xi,e^\uparrow\right]_\C=0
    \]
    for all $\xi\in\Gamma_E(T^{0,1}E)$. Here, $e^\uparrow\in\mx(E)$ is
    identified with $e^\uparrow\in\Gamma(T^{1,0}E)$ via the canonical
    $\C$-linear isomorphism $TE\simeq T^{1,0}E$ as in Remark
    \ref{complexification_rem}.

    Consider therefore the complexification $(TE)_{\mathbb C}$ of $TE$
    (as a vector bundle over $E$).  It is a double vector bundle with
    sides $E$ and $T_{\mathbb C}M$, and with core $E_{\mathbb C}$.
    For $X\in\mx(M)$, $e\in\Gamma(E)$ and $z\in\mathbb C$ the equality
    \[(e\otimes z)^\uparrow=e^\uparrow\otimes z
    \]
    is immediate and it is easy to check that 
    \[ \widehat{\nabla^{\mathbb C}_{X\otimes
          z}}=\widehat{\nabla_X}\otimes z.      
    \]
    Since $\nabla$ is adapted to $J_E$, 
    the respective complexifications satisfy
    \[ J_E^\C\circ \widehat{\nabla^\C_X}=\widehat{\nabla^\C_{J_M^\C X}}
    \]
    for all $X\in \Gamma(T_\C M)$, and in particular
    $\widehat{\nabla^\C_X}\in\Gamma_E(T^{0,1}E)$
    for $X\in\Gamma(T^{0,1}M)$.
    Since $J_E$ is linear over $J_M$,
	$T^{0,1}E$ is a linear subbundle of
    $TE_{\mathbb C}$ over $T^{0,1}M$ and $E$, and with core
    $E^{0,1}$.  It is spanned as a vector bundle over $E$ by the
    sections $\widehat{\nabla^\C_{X}}$ and $e^\uparrow$ for
    $X\in\Gamma(T^{0,1}M)$ and $e\in\Gamma(E^{0,1})$.

    Since the Lie bracket of two core vector fields always vanishes,
    the only equality to check is
    \[ \left[\widehat{\nabla_X^\C}, \theta(e)^\uparrow\right]_\C=0
    \]
    for all $X\in\Gamma(T^{0,1}M)$ and $e\in\Gamma(E)$ a $D^{0,1}$-flat section.  This bracket is easily seen
    to be
      \begin{equation}\label{iis_conn}
        \begin{split}
          \left[\widehat{\nabla^\C_X},
            \theta(e)^\uparrow\right]_\C=\left(\nabla^\C_X\theta(e)\right)^\uparrow=\theta(D_Xe)^\uparrow
          =\theta\left(D^{0,1}_Xe\right)^\uparrow=0. 
        \end{split}
      \end{equation}
      using Remark \ref{complexification_rem}.

 \subsection{Holomorphic Lie algebroids and infinitesimal ideal systems}

 Given a complex manifold $M$, let $\Theta_M$ denote the sheaf of
 holomorphic vector fields on $M$.

Let $A\to M$ be a holomorphic vector bundle and let
$\rho\colon A\to TM$ be a holomorphic vector bundle map, called the
anchor.  Assume that the sheaf $\mathcal{A}$ of holomorphic sections of
$A\to M$ is a sheaf of complex Lie algebras, the anchor map $\rho$
induces a homorphism of sheaves of complex Lie algebras from
$\mathcal A$ to $\Theta_M$, and the Leibniz identity
\[ [X,fY]=\ldr{\rho(X)}f\cdot Y+f [X,Y] \] holds for all
$X,Y\in\mathcal {A}(U)$, $f\in\mathcal O_{M}(U)$, and all open subsets
$U$ of $M$. Then $A$ is a \emph{holomorphic Lie algebroid}, see
e.g.~\cite{LaStXu08} and references therein.

Since the sheaf $\mathcal {A}$ locally generates the
$C^\infty(M)$-module of all smooth sections of $A$, each holomorphic
Lie algebroid structure on a holomorphic vector bundle $A\to M$
determines a unique smooth real Lie algebroid structure on $A$.
Since $\rho\colon A\to TM$ is holomorphic, it satisfies
\[ J_{TM}\circ T\rho=T\rho\circ J_A\,.
\]
Restricting this to the cores 
gives 
$J_M\circ \rho=\rho\circ j_A$, 
which implies $\rho_{\C}(A^{0,1})\subseteq T^{0,1}M$.

Choose a $TM$-connection on $A$ as in Theorem \ref{thm_dec_lin_com}
and its complexification as in Remark \ref{complexification_rem}.
Denote by $\nabla^{0,1}$ the restriction
\[ \nabla^{0,1}\colon\Gamma(T^{0,1}M)\times\Gamma(A^{1,0})\to\Gamma(A^{1,0}).
\]
Recall that by Remark \ref{complexification_rem} it is just $D^{0,1}$
via the canonical $\C$-isomorphism $A\simeq A^{1,0}$. That is, the
$\nabla^{0,1}$-flat sections $a\in\Gamma(A^{1,0})$  are exactly the
holomorphic sections of $A\simeq A^{1,0}$, hence the elements of
$\mathcal A$ via this isomorphism.  If $a,b\in\Gamma_U(A^{1,0})$ are 
$\nabla^{0,1}$-flat (on $U\subseteq M$ open), then
\begin{enumerate}
\item[(i)] $[a,b]_{\mathbb C}$ is again holomorphic so $\nabla^{0,1}$-flat,
\item[(ii)] $\rho(a)\in\Theta_M(U)$ is a holomorphic vector field, so $\bar\partial$-flat,
  where
  \[ \bar\partial\colon\Gamma(T^{0,1}M)\times\Gamma(T^{1,0}M)\to
    \Gamma(T^{1,0}M), \qquad
    \bar\partial_XY=\pr_{T^{1,0}M}[X,Y]_\C
  \]
  is just the $\bar\partial$-operator of the holomorphic vector bundle
  $TM\to M$.
\item[(iii)] $[a,\nu]_\C\in\Gamma(A^{0,1})$ 
for all $\nu\in\Gamma(A^{0,1})$. 
\end{enumerate}

The first two assertions are immediate.
For the third one, consider
$a\in\mathcal A_U$. Then $a$ defines
$a^{1,0}=\half(a-ij(a))\in\Gamma_U(A^{1,0})$ and
$a^{0,1}=\half(a+ij(a))\in\Gamma_U(A^{0,1})$.  Further,
$A^{0,1}_U$ is spanned as a $C^\infty(U)$-module by sections $b^{0,1}$
defined in this manner.
Since for $a,b\in\mathcal A(U)$
\begin{equation*}
 \begin{split}
   [a^{1,0},b^{0,1}]_{\mathbb C}&=\frac{1}{4}[a-j(a)\otimes i,
   b+j(b)\otimes i]_{\mathbb C}
   =\frac{1}{4}\left([a,b]+j[a,b]\otimes i-j[a,b]\otimes i-[a,b]
   \right)=0,
 \end{split}
 \end{equation*}
 where $[a,j(b)]=[a,ib]=i[a,b]=j[a,b]$ follows from the fact that
 $\mathcal A$ is a sheaf of complex Lie algebras.
  Since $A^{1,0}=A_\C/A^{0,1}$, this shows that
  $(T^{0,1}M, A^{0,1}, \nabla^{0,1})$ is a complex infinitesimal
  ideal system \cite{JoOr14} in the complex Lie algebroid $A_\C$.

\begin{thm}
  Let $A\to M$ be a holomorphic vector bundle. If $A\to M$ is a
  holomorphic Lie algebroid then the triple
  $(T^{0,1}M, A^{0,1}, \nabla^{0,1})$ defined as above is an
  infinitesimal ideal system in the complex Lie algebroid $A_\C$.
  \end{thm}

\section{Linear generalised complex structures and Dorfman connections
}
\label{sec_gcs_vb}

Let $E$ be a vector bundle over a smooth manifold $M$. Fix a
skew-symmetric Dorfman connection
\begin{equation}\label{eq_DMC}
	\Delta\colon\Gamma(TM\oplus E^*)\times\Gamma(E\oplus T^*M)
	\rightarrow\Gamma(E\oplus T^*M)\,,
\end{equation}
and therefore a horizontal lift
$\slift\colon \Gamma(TM\oplus E^*)\to \Gamma^\ell_E(TE\oplus
T^*E)$. Consider a double vector bundle morphism
$\JJ\colon TE\oplus T^*E \to TE\oplus T^*E$ as in \eqref{diag_J}.
This section gives conditions on $j$ and $j_C$ for the morphism $\JJ$
to be a generalised complex structure.

\begin{lemma}\label{lem_J_lift}
	Given a double vector bundle morphism $\JJ$ over $j$ as in \eqref{diag_J}
	and a skew-symmetric Dorfman connection $\Delta$ as in \eqref{eq_DMC},
	there is a section $\Phi\in\Gamma((TM\oplus E^*)^*\otimes 
	E^*\otimes(E\oplus T^*M))$ such that for any $\nu\in\Gamma(TM\oplus E^*)$
	\[
		\mathcal{J}\left(\sigma^\Delta(\nu)\right)=\sigma^\Delta(j(\nu))
		+\corelinear{\Phi(\nu)}\,.
	\]
	Here, $(TM\oplus E^*)^*\otimes E^*\otimes(E\oplus T^*M)$ is  identified with 
	$\Hom(TM\oplus E^*,\Hom(E,E\oplus T^*M))$.
\end{lemma}
\begin{proof}
	Since $\mathcal{J}$ is a vector bundle morphism over $j$, 
	$\mathcal{J}\left(\sigma^\Delta(\nu)\right)$ is a linear section over $j(\nu)$. 
	Thus $\mathcal{J}\left(\sigma^\Delta(\nu)\right)-\sigma^\Delta(j(\nu))$ is a core-linear 
	section and this gives for every $\nu$ a section 
	$\Phi(\nu)\in\Gamma(\Hom(E,E\oplus T^*M))$, such that
	\[
		\mathcal{J}\left(\sigma^\Delta(\nu)\right)=\sigma^\Delta(j(\nu))
		+\corelinear{\Phi(\nu)}\,.
	\]
	Since $\JJ$ and $j$ are vector bundle morphisms and
        $\slift$ is a morphism of $C^\infty(M)$-modules, 
	\begin{align*}
          \corelinear{\Phi(f\nu)}&=\JJ\left(\sigma^\Delta(f\nu)\right)-\sigma^\Delta(j(f\nu))
                                   =\JJ\left(q_E^*f\sigma^\Delta(\nu)\right)-\sigma^\Delta(j(f\nu))\\
                                 &=q_E^*f\JJ\left(\sigma^\Delta(\nu)\right)-q_E^*f\sigma^\Delta(j(\nu))
                                   =q_E^*f\corelinear{\Phi(\nu)}=\corelinear{f\Phi(\nu)}
	\end{align*}
        for $f\in C^\infty(M)$ and
        $\nu\in\Gamma\in\Gamma(TM\oplus E^*)$.  That is,
        $\Phi(f\nu)=f\Phi(\nu)$ and
        $\Phi\in\Gamma((TM\oplus E^*)^*\otimes E^*\otimes(E\oplus
        T^*M))$.
\end{proof}

\begin{rmk}\label{lem_J_core_lin}
  In the situation of Lemma \ref{lem_J_lift}, the following equations hold for $\tau\in\Gamma(E\oplus T^*M)$, $\nu\in\Gamma(TM\oplus E^*)$
	and $\varphi\in\Gamma(\Hom(E,E\oplus T^*M))$.
	\begin{enumerate}
        \item $\JJ\left(\tau^\uparrow\right)=j_C(\tau)^\uparrow$. This
          follows directly from the definition of the core morphism
          $j_C$.
        \item $\JJ(\widetilde{\varphi})=\corelinear{j_C\circ\varphi}$. This is an easy calculation using the first equation:
          For $\epsilon\in\Gamma(E^*)$ and $\tau\in\Gamma(E\oplus T^*M)$,
  $\JJ(\widetilde{\epsilon\otimes\tau})=\JJ(\ell_\epsilon\cdot\tau^\uparrow)
    =\ell_\epsilon\cdot\JJ(\tau^\uparrow)=\ell_\epsilon\cdot j_C(\tau)^\uparrow
    =\corelinear{j_C\circ(\epsilon\otimes\tau)}$.
                \end{enumerate}
      \end{rmk}
      
The map $\JJ$ is a generalised complex structure if and only if
\begin{enumerate}
\item $\JJ^2=-\id_{TE\oplus T^*E}$,
\item $\JJ$ is orthogonal, and 
\item the Nijenhuis tensor of $\mathcal J$ vanishes.
\end{enumerate}

The following two lemmas give conditions on $j$, $j_C$ and
$\Phi$ for the map $\JJ\colon TE\oplus T^*E\to TE\oplus
T^*E$ to satisfy the first two properties. The third property is
studied in the next section.

\begin{lemma}\label{lem_J_squ}
  In the situation of Lemma \ref{lem_J_lift}, $\JJ^2=-\id_{TE\oplus T^*E}$ if and only if 
	\begin{enumerate}
		\item $j^2=-\id_{TM\oplus E^*}$\,,
		\item $j_C^2=-\id_{E\oplus T^*M}$\,, 
		\item $\Phi(j(\nu))=-j_C\circ(\Phi(\nu))$ for all 
                  $\nu\in\Gamma(TM\oplus E^*)$.
	\end{enumerate}
\end{lemma}
\begin{proof}
	It is sufficient to check that $\JJ^2=-\id_{TE\oplus T^*E}$ on lifts 
	$\slift(\nu)$ for any $\nu\in\Gamma(TM\oplus E^*)$ and on core sections 
	$\tau^\uparrow$ for any $\tau\in\Gamma(E\oplus T^*M)$, as those sections
	span $\Gamma_{TM\oplus E^*}(TE\oplus T^*E)$ as a $C^\infty(TM\oplus E^*)$-module.
        (2) is immediate by evaluating $\JJ$ on core sections, using Remark
	\ref{lem_J_core_lin}.
        
	For linear sections, the definition of $\Phi$ and Remark
	\ref{lem_J_core_lin} yield 
	\begin{equation}\label{eq_Jsq_lin}\begin{split}
		\JJ^2\left(\slift(\nu)\right)&=\JJ\bigl(\slift(j(\nu))+\corelinear{\Phi(\nu)}\bigr)\\
		&=\slift(j^2(\nu))+\corelinear{\Phi(j(\nu))}
		+\corelinear{j_C\circ(\Phi(\nu))}\,.
	\end{split}\end{equation}
	If $\JJ^2=-\id_{TE\oplus T^*E}$, then the side morphism $j$ has 
	to satisfy $j^2=-\id_{TM\oplus E^*}$. Now \eqref{eq_Jsq_lin} implies
	that 
	\begin{equation}\label{eq_Phi}
		\Phi(j(\nu))=-j_C\circ(\Phi(\nu))\,.
	\end{equation}
	Conversely, if $j^2=-\id_{TM\oplus E^*}$ and 
	$\Phi(j(\nu))=-j_C\circ(\Phi(\nu))$, then \eqref{eq_Jsq_lin} 
	yields immediately  $\JJ^2(\slift(\nu))=-\slift(\nu)$ for all 
	$\nu\in\Gamma(TM\oplus E^*)$.
\end{proof}

\begin{lemma}
	\label{lem_J_ort}
	A double vector bundle morphism $\JJ$ as in \eqref{diag_J} is 
	orthogonal if and only if for any skew-symmetric Dorfman connection 
	$\Delta$ as in \eqref{eq_DMC}:
	\begin{enumerate}
		\item $(j_C)^t=j^{-1}$\,, 
		\item $\Phi(\nu_2)^t(j\nu_1)=-\Phi(\nu_1)^t(j(\nu_2))$ for all 
	$\nu_1,\nu_2\in\Gamma(TM\oplus E^*)$.
	\end{enumerate}
\end{lemma}
\begin{proof}
  Again it is sufficient to check the orthogonality of $\JJ$ on core
  sections and on horizontal lifts. For core sections this is
  immediate, since the pairing of two core sections always
  vanishes. For the pairing of a lift with a core section, use
  \eqref{pairing_core_lin} to obtain that
	\[
		\bigpair{\JJ\slift(\nu)}{\JJ(\tau^\uparrow)}=
		\bigpair{\slift(\nu)}{\tau^\uparrow}\,,
	\]
	if and only if
		$\pair{j(\nu)}{j_C(\tau)}=\pair{\nu}{\tau}$.	
 	Since $\nu$ and $\tau$ are arbitrary, this is equivalent 
 	to $(j_C)^t=j^{-1}$. 
	For the pairing of two lifts compute
	\[
		\bigpair{\JJ(\slift(\nu_1))}{\JJ(\slift(\nu_2))}
		=\ell_{\Phi(\nu_2)^t(j(\nu_1))}
		+\ell_{\Phi(\nu_1)^t(j(\nu_2))}\,,
	\]
	and on the other hand $\pair{\slift(\nu_1)}{\slift(\nu_2)}=0$. 
	Hence $\JJ$ is orthogonal if and only if additionally
	\[
		\Phi(\nu_2)^t(j(\nu_1))=-\Phi(\nu_1)^t(j(\nu_2))\,.\qedhere
	\]
\end{proof}

Define
$\Psi,\Psi^j,(j^*\Psi)\in
\Gamma\bigl((TM\oplus E^*)^*\otimes(TM\oplus E^*)^*\otimes E^*\bigr)$ by
\begin{align*}
	\Psi(\nu_1,\nu_2)&:=\Phi(\nu_1)^t(\nu_2)\,,\quad 
	\Psi^j(\nu_1,\nu_2)&:=\Phi(\nu_1)^t(j\nu_2)\,,\quad 
	(j^*\Psi)(\nu_1,\nu_2)&:=\Psi(j\nu_1,j\nu_2)\,.
\end{align*}
Then Lemma \ref{lem_J_squ} and Lemma 
\ref{lem_J_ort} can be combined as follows
\begin{prop}\label{prop_J_gacs}
	A morphism $\JJ$ as in \eqref{diag_J} is a generalised almost complex 
	structure on $E$, if and only if for every skew-symmetric Dorfman 
	connection $\Delta$ as in \eqref{eq_DMC}:
	\begin{enumerate}
		\item $j^2=-1$\,,
		\item $j=-(j_C)^t$\,,
		\item $\Psi$ is skew-symmetric, i.e.~$\Psi\in\Omega^2(TM\oplus E^*,E^*)$\,,
		\item $\Psi(\nu_1,\nu_2)=-j^*\Psi(\nu_1,\nu_2)$\,.
	\end{enumerate}
      \end{prop}
 	\begin{proof}
 		Under assumption of the properties $j^2=-1$ and $j_C^t=j^{-1}$
 		the condition on $\Phi$ given in Lemma \ref{lem_J_squ} can be reformulated
 		as $\Psi(j\nu_1,\nu_2)=\Psi(\nu_1,j\nu_2)$, whereas the condition
 		on $\Phi$ given in Lemma \ref{lem_J_ort} is then equivalent to 
 		$\Psi(\nu_2,j\nu_1)=-\Psi(\nu_1,j\nu_2)$, in both cases for all 
 		sections $\nu_1,\nu_2$ of $TM\oplus E^*$. Applying the former 
 		equation on the latter and then replacing $\nu_2$ by 
 		$j\nu_3$ shows skew-symmetry of $\Psi$, again using $j^2=-1$. 
 		The former equation is shown to be equivalent to $\Psi=-j^*\Psi$,
 		again by replacing $\nu_2$ with $j\nu_3$.
	 \end{proof}

\subsection{Adapted Dorfman connections} 

This section shows that given a linear generalised almost complex
structure $\JJ$ on $E\to M$, there is a Dorfman connection $\Delta$
which is \textbf{adapted to $\JJ$}, i.e.~such that $\slift$ satisfies
$\JJ(\slift(\nu))=\slift(j\nu)$ for all $\nu\in\Gamma(TM\oplus
E^*)$. Equivalently, the corresponding tensor $\Phi$ defined by Lemma
\ref{lem_J_lift} vanishes.  The choice of such an adapted Dorfman
connection vastly simplifies all the following computations in this
paper.

Recall from Section \ref{sec_Background}
that a change of skew-symmetric Dorfman connection (from $\Delta_1$ to
$\Delta_2$) is equivalent to a 2-form
$\Psi_{12}\in\Omega^2(TM\oplus E^*,E^*)$.
\begin{lemma}\label{lem_change_psi}
  Let $\JJ$ be a linear generalised almost complex structure over $E$
  and choose two skew-symmetric Dorfman connections $\Delta_1$ and $\Delta_2$
  as above with change of splitting $\Psi_{12}$. Then
	\begin{equation}\label{eq_change_psi}
		\Psi_2(\nu_1,\nu_2)
		=\Psi_1(\nu_1,\nu_2)
		-\Psi_{12}(\nu_1,j\nu_2)
		-\Psi_{12}(j\nu_1,\nu_2)\,,
	\end{equation}
	where $\Psi_1,\Psi_2$ are the 2-forms defined as in  Proposition 
	\ref{prop_J_gacs} by $\JJ$  and $\Delta_1$ and $\Delta_2$, respectively. 
\end{lemma}
\begin{proof}
A computation yields
	\begin{equation*}
		\corelinear{\Phi_2(\nu)}
		=\JJ(\sigma_2(\nu))
		-\sigma_2(j\nu)\\
		=\corelinear{\Phi_1(\nu)}
		+\corelinear{j_C\circ\Phi_{12}(\nu)}
		-\corelinear{\Phi_{12}(j\nu)}\,.
	\end{equation*}
	Dualising this equality leads to 
	\begin{equation*}
		\Psi_2(\nu_1,\nu_2)
		=\Psi_1(\nu_1,\nu_2)
		-\Psi_{12}(\nu_1,j\nu_2)
		-\Psi_{12}(j\nu_1,\nu_2)\,,
	\end{equation*}
	using  $j_C^t=-j$, see Proposition \ref{prop_J_gacs}. 
\end{proof}

Now  \eqref{eq_change_psi}  is used to find the existence of a Dorfman connection adapted to $\JJ$.

\begin{prop}\label{prop_adapted_DMC}
  For every linear generalised complex structure $\JJ$ on $E$ there is a
  skew-symmetric $(TM\oplus E^*)$-Dorfman connection $\Delta$ on
  $E^*\oplus TM$ such that $\JJ(\slift(\nu))=\slift(j\nu)$ for all
  $\nu\in\Gamma(TM\oplus E^*)$.
\end{prop}
\begin{proof}
  Fix any skew-symmetric Dorfman connection $TM\oplus E^*$-Dorfman
  connection $\Delta_1$ on $E^*\oplus TM$ and denote the corresponding
  lift by $\sigma_1$. Proposition \ref{prop_J_gacs} defines a two-form
  $\Psi_1\in\Omega^2(TM\oplus E^*,E^*)$ such that
  $\JJ(\sigma_1(\nu))=\sigma_1(j\nu)+\corelinear{\Psi_1(\nu,\cdot)}$.
  Let now $\Psi_{12}(\nu_1,\nu_2):=-\half\Psi_1(\nu_1,j\nu_2)$.  By
  Proposition \ref{prop_J_gacs} this form is skew-symmetric and
  therefore the dull bracket defined by \eqref{eq_change_bra} is
  skew-symmetric again. Now according to \eqref{eq_change_psi} and
  using the properties from Proposition \ref{prop_J_gacs} the
  corresponding 2-form $\Psi_2$ vanishes. Hence the Dorfman connection
  Dorfman connection $\Delta:=\Delta_2$ satisfies $\Psi_2=0$. By
  definition of $\Psi_2$, this is equivalent to
  $\JJ\circ\slift=\slift\circ j$.
\end{proof}

\medskip

The remainder of this section characterises the set of Dorfman connections
adapted to $\JJ$.
\begin{mydef}\label{def_j_equiv}
  Let $E\to M$ be a smooth vector bundle and let
  $j\colon TM\oplus E^*\to TM\oplus E^*$ be an arbitrary vector bundle
  morphism.  Two skew-symmetric $(TM\oplus E^*)$-Dorfman connection
  $\Delta_1$ and $\Delta_2$ on $E^*\oplus TM$ are \textbf{$j$-equivalent},
  if their change of splittings $\Psi_{12}$ defined by
  \eqref{eq_change_bra} satisfies
	\[
		\Psi_{12}(\nu_1,\nu_2)=\Psi_{12}(j\nu_1,j\nu_2)
    \]
              for all
  $\nu_1,\nu_2\in\Gamma(TM\oplus E^*)$.
\end{mydef}

The following lemma shows that if a Dorfman connection
$\Delta_1$ is adapted to $\JJ$, then a second
Dorfman connection $\Delta_2$ is adapted to $\JJ$ if and only if they 
are $j$-equivalent. 
\begin{lemma}\label{lem_j_equiv}
	Let $\JJ$ be a linear generalised almost complex structure $\JJ$ 
	on a vector bundle $E\to M$ and $\Delta_1$ and $\Delta_2$ be two 
	skew-symmetric $(TM\oplus E^*)$-Dorfman connections on $E\oplus T^*M$. 
	Denote the two-forms given by Proposition \ref{prop_J_gacs} 
	corresponding to $\Delta_1$ and $\Delta_2$ by $\Psi_1$ and $\Psi_2$, 
	respectively. Then $\Delta_1$ and $\Delta_2$ are $j$-equivalent if 
	and only if $\Psi_1=\Psi_2$.
\end{lemma}
\begin{proof}
	Denote the change of splittings again by  $\Psi_{12}$.  
	 $\Delta_1$ is $j$-equivalent to $\Delta_2$ if and only if
	for all $\nu_1,\nu_2\in\Gamma(TM\oplus E^*)$
	\begin{equation}\label{eq_j_equiv}
	\Psi_{12}(\nu_1,j\nu_2)=-\Psi_{12}(j(j\nu_1),j\nu_2)=-\Psi_{12}(j\nu_1,\nu_2)\,.
	\end{equation}
	By \eqref{eq_change_psi}, $\Psi_1$ and $\Psi_2$
	are related by
	\[
	\Psi_2(\nu_1,\nu_2)
	=\Psi_1(\nu_1,\nu_2)
	-\Psi_{12}(\nu_1,j\nu_2)
	-\Psi_{12}(j\nu_1,\nu_2)\,.
	\]
	So $\Psi_1=\Psi_2$ if and only if \eqref{eq_j_equiv} holds, that
	is, if and only if $\Delta_1$ and $\Delta_2$ are $j$-equivalent. 	
      \end{proof}

\subsection{Integrability}\label{sec_lgcs_int}
Consider a linear generalised almost complex structure $\JJ$ 
on $E$ as in \eqref{diag_J} and fix a skew-symmetric Dorfman connection
$\Delta$ as in \eqref{eq_DMC}, which is adapted to $\JJ$. In particular 
$j$, $j_C$ satisfy the conditions of Proposition \ref{prop_J_gacs}, and
$\JJ\bigl(\slift(\nu)\bigr)=\slift(j\nu)$ for all 
$\nu\in\Gamma(TM\oplus E^*)$.
Evaluated at two core sections the Nijenhuis tensor of $\JJ$ vanishes
trivially, since the Courant-Dorfman bracket of two core sections
vanishes and the double vector bundle morphism $\JJ$ sends core
sections to core sections.

For the Nijenhuis tensor of $\JJ$ evaluated at a horizontal lift 
$\slift(\nu)$ for $\nu\in\Gamma(TM\oplus E^*)$ and a core section 
$\tau^\uparrow$ for $\tau\in\Gamma(E\oplus T^*M)$, compute 
\[
	N_\JJ\bigl(\slift(\nu),\tau^\uparrow\bigr)
	=\bigl(\Delta_\nu\tau\bigr)^\uparrow
	-\bigl(\Delta_{j(\nu)}j_C(\tau)\bigr)^\uparrow
	+\bigl(j_C(\Delta_{j(\nu)}\tau)\bigr)^\uparrow
	+\bigl(j_C(\Delta_\nu j_C(\tau))\bigr)^\uparrow\,.
\]
Thus the Nijenhuis tensor of $\JJ$ vanishes for any such pair 
of a lift $\slift(\nu)$ and a core section $\tau^\uparrow$ if 
and only if for all $\nu\in\Gamma(TM\oplus E^*)$ and 
$\tau\in\Gamma(E\oplus T^*M)$ 
\begin{equation}\label{eq_Nij_lin_core_dual}
	\Delta_\nu\tau
	-\Delta_{j(\nu)}j_C(\tau)
	+j_C\bigl(\Delta_{j(\nu)}\tau\bigr)
	+j_C\bigl(\Delta_\nu j_C(\tau)\bigr)=0\,.
\end{equation}

As the pairing is non-degenerate, this condition can be dualised by
pairing it with a second section $\nu_2\in\Gamma(TM\oplus
E^*)$. Recall that $\Delta$ is dual to a dull
bracket $\DMbra{\cdot}{\cdot}_{\Delta}$. Then the properties of $j$
and $j_C$ obtained in Proposition \ref{prop_J_gacs} lead to
\[
	\Bigl\langle\Delta_{\nu_1}\tau
	-\Delta_{j(\nu_1)}j_C(\tau)
	+j_C\bigl(\Delta_{j(\nu_1)}\tau\bigr)
	+j_C\bigl(\Delta_{\nu_1} j_C(\tau)\bigr),\nu_2\biggr\rangle
	=\Bigpair{\tau}{-N_{j,\DMbra{\cdot}{\cdot}_{\Delta}}(\nu_1,\nu_2)}\,.
\]

Thus the Nijenhuis tensor of a generalised almost complex structure 
$\JJ$ vanishes when evaluated at a pair of any lift $\slift(\nu)$ 
and any core section $\tau^\uparrow$ if and only if the Nijenhuis 
tensor of $j$ with respect to the dull bracket 
$\DMbra{\cdot}{\cdot}_\Delta$ vanishes. 

Finally, compute for
$\nu_1,\nu_2\in\Gamma(TM\oplus E^*)$:
\begingroup
\allowdisplaybreaks[1]
\begin{align*}
  N_\JJ\bigl(\slift(\nu_1),\slift(\nu_2)\bigr)
  &=\slift(N_{j,\DMbra{\cdot}{\cdot}_\Delta}(\nu_1,\nu_2))+\corelinear{R_\Delta(j(\nu_1),j(\nu_2))(\cdot,0)}
    -\corelinear{R_\Delta(\nu_1,\nu_2)(\cdot,0)}\\*
  &\quad-\corelinear{j_C\circ R_\Delta(j(\nu_1),\nu_2)(\cdot,0)}
    -\corelinear{j_C\circ R_\Delta(\nu_1,j(\nu_2))(\cdot,0)}\,.
\end{align*}

Recall that since the dull bracket on $TM\oplus E^*$ is anchored by
$\pr_{TM}$ and $\Delta_{(X,\eps)}(0,\theta)=(0,\LL_X\theta)$ the
curvature $R_{\Delta}(\nu_1,\nu_2)(0,\theta)$ for
$\theta\in\Gamma(T^*M)$ always vanishes. Therefore the terms with
$R_\Delta$ above evaluated at $(e,0)$ vanish if and only if the
corresponding terms vanish evaluated at $(e,\theta)$ for any
$\theta \in\Gamma(T^*M)$.  Thus the Nijenhuis tensor of $\JJ$ vanishes
for all sections if and only if the Nijenhuis tensor of $j$ with
respect to $\DMbra{\cdot}{\cdot}_\Delta$ vanishes and additionally the
curvature of the adapted $\Delta$ satisfies
\[
	0=R_\Delta\bigl(j(\nu_1),j(\nu_2)\bigr)(\tau)
	-R_\Delta(\nu_1,\nu_2)(\tau)
	-j_C\Bigl( R_\Delta(j(\nu_1),\nu_2)(\tau)\Bigr)
	-j_C\Bigl(R_\Delta(\nu_1,j(\nu_2))(\tau)\Bigr)\,,
\]
for all $\nu_1,\nu_2\in\Gamma(TM\oplus E^*)$ and 
$\tau\in\Gamma(E\oplus T^*M)$. By pairing the right hand side of 
this equation with a third section $\nu_3$ of $TM\oplus E$,
obtain equivalently
\begin{equation}\label{eq_AAA_Jacobi}
\begin{split}
	0&=\Jac_\Delta(j\nu_1,j\nu_2,\nu_3)
	+\Jac_\Delta(j\nu_1,\nu_2,j\nu_3)
	+\Jac_\Delta(\nu_1,j\nu_2,j\nu_3)
	-\Jac_\Delta(\nu_1,\nu_2,\nu_3)\\
	&=\DMbra{\nu_1}{\DMbra{\nu_2}{\nu_3}_\Delta}_\Delta
	-\DMbra{j\nu_1}{\DMbra{j\nu_2}{\nu_3}_\Delta}_\Delta
	-\DMbra{j\nu_1}{\DMbra{\nu_2}{j\nu_3}_\Delta}_\Delta
	-\DMbra{\nu_1}{\DMbra{j\nu_2}{j\nu_3}_\Delta}_\Delta\\
	&\quad+\text{cyclic permutations in 1,2,3}\\
	&=\DMbra{\nu_1}{\DMbra{\nu_2}{\nu_3}_\Delta
	-\DMbra{j\nu_2}{j\nu_3}_\Delta}_\Delta
	+\DMbra{j\nu_1}{-\DMbra{\nu_2}{j\nu_3}_\Delta
	-\DMbra{j\nu_2}{\nu_3}_\Delta}_\Delta\\
	&\quad+\text{cyclic permutations in 1,2,3}\,.
\end{split}
\end{equation}

Define a bracket $\AAA$ on $\Gamma(TM\oplus E^*)$ by 
\begin{equation}\label{eq_def_AAA}
	\AAA(\nu_1,\nu_2)
	:=\half\Bigl(\DMbra{\nu_1}{\nu_2}_\Delta
	-\DMbra{j\nu_1}{j\nu_2}_\Delta\Bigr)\,.
\end{equation}
Then $N_{j,\DMbra{\cdot}{\cdot}_\Delta}$ vanishes if and only if 
$\AAA$ satisfies
\begin{equation}\label{AClinear}
	\AAA(\nu_1,j\nu_2)=j\AAA(\nu_1,\nu_2)\,.
\end{equation}
Furthermore, \eqref{eq_AAA_Jacobi} is 
equivalent to the Jacobi identity of this bracket $\AAA$.

Note that the bracket $\AAA$ does not admit a $TM$-valued anchor on 
$TM\oplus E^*$, and is thus not a (real) Lie algebroid bracket on 
$TM\oplus E^*$, since
\[
	\AAA(\nu_1,f\nu_2)
	=f\AAA(\nu_1,\nu_2)
	-\half\pr_{TM}(j\nu_1)(f)j\nu_2
	+\half\pr_{TM}(\nu_1)(f)\nu_2 
      \]
      for $f\in C^\infty(M)$ and $\nu_1,\nu_2\in\Gamma(TM\oplus E^*)$.
      However, the map $j$ is a fibre-wise complex structure in 
      $TM\oplus E^*$, with respect to which $\AAA$ is $\C$-bilinear 
      by \eqref{AClinear} and skew-symmetric. And the equation 
      above shows that with the complex anchor 
      $\rho\colon TM\oplus E^*\to T_\C M$ defined by
\begin{equation}\label{eq_anchor_cplx_LA}
	\rho(\nu):=\pr_{T_\C M}\left(\half(\nu\otimes 1-(j\nu)\otimes i)\right)\,,
\end{equation}
$(TM\oplus E^*,\rho,\AAA)$ is a complex Lie algebroid as defined 
in \cite{Weinstein07}. Furthermore, it is easy to check that 
$\AAA$ is independent of the choice of adapted Dorfman connection 
$\Delta$.
\begin{thm}\label{thm_lgcs_DMC}
        Let $E\to M$ be a smooth vector bundle.
        A linear generalised complex structure on $E$ is equivalent to
        \begin{enumerate}
        \item A vector bundle morphism
          $j\colon TM\oplus E^*\to TM\oplus E^*$ such that
          $j^2=-\id_{TM\oplus E^*}$, i.e.~a complex structure in the fibers of $TM\oplus E^*$, and
        \item a choice of $j$-equivalence class of Dorfman
          $TM\oplus E^*$-connections on $E\oplus T^*M$ such that
          $TM\oplus E^*\to M$, equipped with the bracket $\AAA$
          defined in \eqref{eq_def_AAA} by the corresponding dull
          brackets and the anchor $\rho$ in \eqref{eq_anchor_cplx_LA}
          becomes a complex Lie algebroid.
          \end{enumerate}
        \end{thm}
        
\begin{proof}
  Given a linear generalised almost complex structure $\JJ$ on 
  $E$ as in \eqref{diag_J}, Proposition \ref{prop_J_gacs} and 
  Proposition \ref{prop_adapted_DMC} define the vector bundle 
  morphism $j$ and an adapted Dorfman connection $\Delta$, 
  in turn uniquely defining the bracket $\AAA$ on 
  $\Gamma(TM\oplus E^*)$ by \eqref{eq_def_AAA}.
  By the arguments above, integrability of $\JJ$ implies that
  $\AAA$ is a complex Lie algebroid bracket with anchor given
  by \eqref{eq_anchor_cplx_LA}. 
	
	Conversely, given $j$ and a complex Lie algebroid structure
	where the bracket $\AAA$ comes from a Dorfman connection 
	$\Delta$ as above, define a double vector bundle morphism 
	$\JJ\colon TE\oplus T^*E\to TE\oplus T^*E$ by  
	\begin{align*}
		\JJ(\tau^\uparrow)&:=(-j^t(\tau))^\uparrow\,,\qquad
		\JJ(\slift(\nu)):=\slift(j(\nu))\,
	\end{align*}
        for $\tau\in\Gamma(E\oplus T^*M)$ and
        $\nu\in\Gamma(TM\oplus E^*)$.  Again by Proposition
        \ref{prop_J_gacs} and the computations above of the Nijenhuis
        tensor of $\JJ$ on linear and core sections, this defines a
        linear generalised complex structure on $E$.

        These two
        constructions are inverse to each other since neither $\JJ$
        nor $\AAA$ depend on the choice of the Dorfman connection in
        the $j$-equivalence class of $\Delta$. 
      \end{proof}

\subsection{Generalised K{\"a}hler structures on vector bundles}
Generalised K{\"a}hler structures were introduced in
\cite{Gualtieri04,Gualtieri14}.  
Any automorphism $G$ of $TM\oplus T^*M$ which is symmetric $G^t=G$,
and squares to $\id$ defines a symmetric metric on $TM\oplus
T^*M$. Such an automorphism is therefore called a \emph{metric}. 
\begin{mydef}
	A generalised K{\"a}hler structure on a manifold is a pair of 
	commuting generalised complex structures $\JJ_1$ and $\JJ_2$ 
	such that the symmetric non-degenerate metric $G:=-\JJ_1\circ\JJ_2$ 
	is positive definite. 
      \end{mydef}
      This section shows that in the case of a generalised K{\"a}hler
      structure on a vector bundle $E$ there exists a Dorfman
      connection which is adapted in the sense of Proposition
      \ref{prop_adapted_DMC} to both generalised complex structures
      simultaneously.  Take a vector bundle $E\to M$ equipped with a
      \emph{linear} generalised K{\"a}hler structure, i.e.~$\JJ_1$ and
      $\JJ_2$ are both linear. Denote the side morphisms on
      $TM\oplus E^*$ by $j_1$ and $j_2$, respectively. Take now any
      skew-symmetric $TM\oplus E^*$-Dorfman connection $\Delta$ on
      $E\oplus T^*M$. This gives rise to the corresponding $2$-forms
      $\Psi_1$ and $\Psi_2$ as in Proposition \ref{prop_J_gacs}.
\begin{lemma}\label{lem_kahler_psi}
	In the setting above, for all $\nu_1,\nu_2\in TM\oplus E^*$
	\[
	\Psi_2(j_1\nu_1,\nu_2)+\Psi_2(\nu_1,j_1\nu_2)
	=\Psi_1(j_2\nu_1,\nu_2)+\Psi_1(\nu_1,j_2\nu_2)\,.
	\]
\end{lemma}
\begin{proof}
  Since $\JJ_1$ and $\JJ_2$ commute, so do their side morphisms $j_1$
  and $j_2$. Thus $\sigma^\Delta(j_1j_2\nu)=\sigma^\Delta(j_2j_1\nu)$
  and with the definition of $\Phi_1$ and $\Phi_2$ in Lemma
  \ref{lem_J_lift} this leads by straightforward computation to the
  equality
  $j_1^t\circ\Phi_2(\nu)-\Phi_1(j_2\nu)=j_2^t\circ\Phi_1(\nu)-\Phi_2(j_1\nu)$.
  Pairing this with a second arbitrary section
  $\nu_2\in\Gamma(TM\oplus E^*)$ and dualising then gives the desired
  equality. 
\end{proof}
This lemma easily yields the existence of a Dorfman connection adapted
to both generalised complex structures.
\begin{prop}\label{prop_kahler_adapted}
	Given two commuting linear generalised complex structures $\JJ_1$ and 
	$\JJ_2$ on a vector bundle $E\to M$, there is a $TM\oplus E^*$-Dorfman 
	connection $\Delta$	on $E\oplus T^*M$ which is adapted to both $\JJ_1$
	and $\JJ_2$.
\end{prop}
\begin{proof}
	Take a skew-symmetric Dorfman connection $\Delta_1$, which is adapted
	to $\JJ_1$ as constructed in Proposition \ref{prop_adapted_DMC}. Thus 
	the 2-form $\Psi_1$ vanishes. The previous Lemma \ref{lem_kahler_psi}
	then shows that $\Psi_2(\nu_1,\nu_2)=\Psi_2(j_1\nu_1,j\nu_2)$ for all
	$\nu_1,\nu_2\in\Gamma(TM\oplus E^*)$. In order to obtain a 
	Dorfman connection $\Delta_2$ adapted to $\JJ_2$, use as 
	change of splitting the form $\Psi_{12}:=-\half\Psi_2(\cdot,j_2\cdot)$ 
	as shown in the proof of Proposition \ref{prop_adapted_DMC}. But since
	$j_1$ and $j_2$ commute, also this change of 
	splittings satisfies $\Psi_{12}(j_1\nu_1,j_1\nu_2)=\Psi_{12}(\nu_1,\nu_2)$ 
	for all $\nu_1,\nu_2\in TM\oplus E^*$. Thus $\Delta_2$ is $j_1$-equivalent 
	to $\Delta_1$ as in Definition \ref{def_j_equiv}. Now Lemma	\ref{lem_j_equiv}
	shows that $\Delta_2$ is still adapted to $\JJ_1$ and therefore to both
	generalised complex structures simultaneously.
\end{proof}

\subsection{Complex VB-Dirac structures}
\label{sec_cplx_VB_Dirac}

This section shows that a linear generalised complex structure on a
vector bundle $E$ is equivalent to a pair of complex conjugated
VB-Dirac structures.  Consider the complexification of $\TT E$ as a
vector bundle over $E$. This is again a double vector bundle
$\TT_\C E\cong T_\C E\oplus T_\C ^*E$ with complexified core and side
bundle.
\[{\footnotesize
	\begin{tikzcd}
	T_\C E\oplus T_\C ^*E \ar[rr] \ar[dd,"(\pi_{TM\oplus E^*})_\C"'] & & E \ar[dd]\\
	 & E_\C\oplus T_\C^*M \ar[rd] & \\
	T_\C M\oplus E_\C^* \ar[rr] & & M
	\end{tikzcd} 
	\,.}
      \]

It is straightforward to extend the results of \cite{Jotz18a} 
complex linearly to characterize linear splittings of the double
vector bundle $\TT_\C E$ which are additionally $\C$-linear over
$E$, giving a correspondence between such splittings and complex
$(T_\C M\oplus E^*_\C)$-Dorfman connections on $E_\C\oplus T^*_\C M$. 
      
Furthermore, the splitting corresponding to the complexification 
$\Delta^\C$ of a real Dorfman connection $\Delta$ is the 
complexification of the splitting $\Sigma^\Delta$, i.e.~the
complex linear extension in $TM\oplus E^*$. 

Let $\JJ\colon\TT E\rightarrow \TT E$ be a linear generalised complex
structure on $E$ over $j\colon TM\oplus E^*\rightarrow TM\oplus E^*$
and with core morphism $j_C$ as in Definition \ref{def_lin_gcs}.  The
$\pm i$-eigenbundles of the complexification
$\JJ_\C\colon \TT_\C E\rightarrow \TT_\C E$ of $\JJ$ build a pair of
complex conjugated complex Dirac structures $D_{\pm}\subset \TT_\C E$:
\[
	D_{+}=\{\xi-i\JJ(\xi)|\,\xi\in\TT E\}\qquad
	D_{-}=\{\xi+i\JJ(\xi)|\,\xi\in\TT E\} \,.
\]

It is easy to see that these are two sub-double vector bundles 
of $\TT_\C E$
\[
	\begin{tikzcd}
	D_{\pm} \ar[r] \ar[d,"(\pi_{TM\oplus E^*})_\C|_{D_{\pm}}"'] & E \ar[d,"q_E"]\\
	U_{\pm} \ar[r] & M
	\end{tikzcd} 
	\,,
\]
where $U_{\pm}$ is the $\pm i$-eigenbundle of the complexification 
$j_\C\colon(TM\oplus E^*)_\C\to (TM\oplus E^*)_\C$ of $j$. The core
of this double vector bundle is $K_\pm\subseteq (E\oplus T^*M)_\C$, 
the $\pm i$-eigenbundle of the complexification $j_{C,\C}$ of $j_C$.
\begin{lemma}\label{lem_ann_U}
	In the situation above,
	\[
		U_{\pm}^\circ=K_{\pm}\,,
	\]
	where $U_{\pm}^\circ$ denotes the annihilator of $U_\pm$ in 
	$(TM\oplus E^*)^*\cong E\oplus T^*M$. 
\end{lemma}
\begin{proof}
	This follows directly from the property $j^t=-j_C$ of 
	Proposition \ref{prop_J_gacs}.
\end{proof}

Now consider the complex linear extension $\AAA_\C$ of $\AAA$ defined in
\eqref{eq_def_AAA}. 
\begin{prop}\label{prop_AAA_U_lie}
  The restriction to $U_\pm$ of the complexified bracket $\AAA_\C$ 
	coincides with the restriction of the dull bracket 
	$\DMbra{\cdot}{\cdot}_{\Delta^\C}$ corresponding to the
	complexification of an adapted Dorfman connection $\Delta$
	and defines a $\C$-Lie algebroid structure on $U_\pm$ with 
	anchor $\pr_{T_\C M}|_{U_\pm}$.
	
	It isomorphic as complex Lie algebroid to 
	$(TM\oplus E^*,\rho,\AAA)$ of Theorem \ref{thm_lgcs_DMC} via
	the canonical isomorphism $\nu\mapsto \frac{\nu\mp i j \nu}{2}$. 
\end{prop}
\begin{proof}
  By \eqref{AClinear} the complexified bracket
  $\AAA_\C$ restricts to the two eigenbundles $U_{\pm}$ of $j_\C$.
		
	It follows directly from the definition of $\AAA$ in \eqref{eq_def_AAA}
	that the restriction of $\AAA_\C$ to $U_\pm$ coincides with the 
	complexification of $\DMbra{\cdot}{\cdot}_\Delta$ for any adapted 
	Dorfman connection $\Delta$. Therefore it is anchored by the 
	restriction of $\pr_{T_\C M}$ and satisfies the Leibniz identity.
	Skew-symmetry and Jacobi identity follow from the corresponding 
	properties for $\AAA$.
	
	It is easy to check that $TM\oplus E^*\to U_\pm$, 
	$\nu\mapsto \nu\mp i j \nu$	is indeed an isomorphism of complex 
	Lie algebroids, where the fibre-wise complex structures are $j$
	on one side and induced from the complexification on the other. 
      \end{proof}
Proposition \ref{prop_AAA_U_lie}  and Theorem
        \ref{thm_lgcs_DMC} now imply Theorem \ref{main}.
      \begin{proof}[Proof of Theorem \ref{main}]
        The complex Lie algebroid in induced as in Theorem
        \ref{thm_lgcs_DMC} by a linear generalised complex structure
        is quasi-real by Proposition \ref{prop_AAA_U_lie}.

        Conversely, let $E\to M$ be a smooth vector bundle with a
        vector bundle morphism $j\colon TM\oplus E^*\to TM\oplus E^*$
        such that $j^{2}=-\id_{TM\oplus E^*}$.  Assume that
        $(TM\oplus E^*, \pr_{TM},j,[\cdot\,,\cdot])$ is a quasi-real
        complex Lie algebroid as in Definition \ref{def_qrca}.  Then
        there is a dull bracket $\llb\cdot\,,\cdot\rrb$ on sections of
        $TM\oplus E^*$, that is anchored by $\pr_{TM}$ and such that
        the canonical isomorphism
        $TM\oplus E^*\to (TM\oplus E^*)^{1,0}$ is an isomorphism of
        the complex \emph{Lie algebroids}
        $((TM\oplus E^*)^{1,0}, \pr_{TM}^{1,0}, j_\C,
        \llb\cdot\,,\cdot\rrb_{1,0})$ and
        $(TM\oplus E^*, \pr_{TM},j,[\cdot\,,\cdot])$.  The
        compatibility of the anchors is immediate by definition of
        $\pr_{TM,j}$.

        Then for all $\nu_1,\nu_2\in\Gamma(TM\oplus E^*)$,
        \begin{equation}\label{quasi_real}
          \half\left([\nu_1,\nu_2]\otimes 1-j[\nu_1,\nu_2]\otimes i\right)
            =\frac{1}{4}\llb \nu_1\otimes 1-j(\nu_1)\otimes i, \nu_2\otimes 1- j(\nu_2)\otimes i\rrb.
          \end{equation}
          The right-hand side of this equation is \[
            \half\left(\frac{\llb \nu_1,\nu_2\rrb-\llb j\nu_1, j\nu_2\rrb}{2}\otimes 1-\frac{\llb j\nu_1, \nu_2\rrb+\llb \nu_1, j\nu_2\rrb}{2}\otimes i\right),
          \]
          which, compared with its left-hand side, gives
          \[ \left\{\begin{array}{ll}
                      [\nu_1,\nu_2]&=\frac{\llb \nu_1,\nu_2\rrb-\llb j\nu_1, j\nu_2\rrb}{2}\\
                      j[\nu_1,\nu_2]&=\frac{\llb j\nu_1, \nu_2\rrb+\llb \nu_1, j\nu_2\rrb}{2}
                      \end{array}\right..
                  \]
                  Since $\nu_1,\nu_2\in\Gamma(TM\oplus E^*)$ were
                  arbitrary, these equations yield that
                  $\llb\cdot\,,\cdot\rrb$ has vanishing Nijenhuis tensor
                  with respect to $j$: $N_{j,\llb\cdot\,,\cdot\rrb=0}$
                  and $[\cdot\,,\cdot]$ is defined by
                  $\llb\cdot\,,\cdot\rrb$ and $j$ as in
                  \eqref{eq_def_AAA}.  

                  Finally, assume that $\llb\cdot\,,\cdot\rrb$ and
                  $\llb\cdot\,,\cdot\rrb'$ both satisfy
                  \eqref{quasi_real} and consider the form $\Psi\in\Omega^2(TM\oplus E^*,E^*)$ defined by
                  \[ \llb\nu_1,\nu_2\rrb'=\llb\nu_1,\nu_2\rrb+(0,\Psi(\nu_1,\nu_2))
                  \]
                  for all $\nu_1,\nu_2\in\Gamma(TM\oplus E^*)$, see
                  \eqref{eq_change_bra}. Then
                  \[\frac{\llb \nu_1,\nu_2\rrb'-\llb j\nu_1, j\nu_2\rrb'}{2}=[\nu_1,\nu_2]=\frac{\llb \nu_1,\nu_2\rrb-\llb j\nu_1, j\nu_2\rrb}{2}
                  \]
                  and so
                  \begin{equation*}
                    \begin{split}
                      \llb\nu_1,\nu_2\rrb-\llb j\nu_1, j\nu_2\rrb&=\llb \nu_1,\nu_2\rrb'-\llb j\nu_1, j\nu_2\rrb'\\
                      &=\llb\nu_1,\nu_2\rrb+(0,\Psi(\nu_1,\nu_2))-\llb j\nu_1, j\nu_2\rrb-(0,\Psi(j\nu_1,j\nu_2)),
                      \end{split}
                    \end{equation*}
                    which implies
                    $\Psi(\nu_1,\nu_2)=\Psi(j\nu_1,j\nu_2)$ for all
                    $\nu_1,\nu_2\in\Gamma(TM\oplus E^*)$.  Hence, the
                    Dorfman connections defined by the two dull
                    brackets are $j$-equivalent as in Defintion
                    \ref{def_j_equiv}.

                    By Theorem \ref{thm_lgcs_DMC},
                  $E$ is thus equipped with a linear generalised complex
                  structure $\mathcal J$ with core $j$ and such that
                  $\llb\cdot\,,\cdot\rrb$ is adapted to $\mathcal J$.
                \end{proof}

                The following example discusses how Theorem \ref{main}
                specialises in the case of holomorphic vector bundles.
                \begin{example}\label{hol_ex_cpla}
                   Let $E\to M$ be a holomorphic vector bundle. As shown in 
                   Section \ref{sec_hol_vb} this corresponds
        to a linear complex structure $J_E$ on $E$ over
        $J_M\colon TM\to TM$ with core morphism $j_E\colon E\to E$.
        The corresponding generalised complex structure $\JJ$, its
        side morphism $j$ and core morphism $j_C$ are
\begin{equation}\label{eq_complex_gcs}
	\JJ=
	\left(\begin{array}{cc} J_E & 0 \\ 
	0 & -J_E^t \end{array}\right)\,,\qquad 
	j=
	\left(\begin{array}{cc} J_M & 0 \\ 
	0 & -j_E^t \end{array}\right)\,,\qquad
	j_C=
	\left(\begin{array}{cc} j_E & 0 \\ 
	0 & -J_M^t \end{array}\right)\,.
\end{equation}
The equality $j=-j_C^t$ is immediate.

The eigenbundles of the complexified morphisms are  given by 
\begin{equation}\label{eq_cplx_U_K}
	\begin{split} 
		U_+&=T^{1,0}M\oplus (E^{0,1})^*\,,\qquad
		K_+=E^{1,0}\oplus (T^{0,1}M)^*\,,\\
		U_-&=T^{0,1}M\oplus (E^{1,0})^*\,,\qquad
		K_-=E^{0,1}\oplus (T^{1,0}M)^*\,,\\
	\end{split}
\end{equation}
where $T^{1,0}M$, $T^{0,1}M$, $E^{1,0}$ and $E^{0,1}$ are
the $\pm i$-eigenbundles of $J_M$ and $j_E$, respectively.

   A linear generalised complex structure on $E$ is
        simply a holomorphic structure on $E$ if and only if the 
        generalised complex structure is of the form above, see Section 
        \ref{sec_hol_vb}. In that case the complex Lie algebroid structure 
        on $TM\oplus E^*$ with fibrewise complex structure $j=(J_M,-j_E^t)$ 
        is given by the anchor
            \[ \rho\colon TM\oplus E^*\to T^{1,0}M\subseteq T_{\mathbb C}M, 
            \qquad \rho(X,\eps)=\half(X\otimes
                1-(J_MX)\otimes i)\,,
            \]
            and,  by \eqref{eq_std_bra}, the bracket
            \[\AAA\bigl((X_1,\eps_1), (X_2,\eps_2)\bigr)
                      =\half\Bigl([X_1,X_2]-[J_MX_1,J_MX_2],
               \nabla_{X_1}^*\eps_2-\nabla_{X_2}^*\eps_1+\nabla_{J_MX_1}^*j_E^*\eps_2
               -\nabla_{J_MX_2}^*j_E^*\eps_1
              \Bigr)
            \]
            where $\nabla^*$ is a flat $TM$-connection on $E^*$, dual
            to a $TM$-connection $\nabla$ on $E$ which is adapted to
            the linear complex structure $J_E$, see Section
            \ref{sec_hol_vb}.
            
            Now $TM$ is $\C$-isomorphic to $T^{1,0}M$ and $E^*$ is $\C$-isomorphic 
            to $(E^{0,1})^*$, since the complex structure on $E^*$ is taken to 
            be $-j_E^t$. After these identifications, $\nabla^*$ is dual to the 
            $T^{1,0}M$-connection $\nabla$ on $E^{0,1}$, defined as complex conjugate 
            to the flat $T^{0,1}M$-connection $\bar\partial$ on $E^{1,0}\cong E$
            corresponding to the holomorphic structure. That is for 
            $X\in\Gamma(T^{1,0}M)$ and $e\in\Gamma(E^{0,1})$, 
            \[
            \nabla_Xe:=\overline{\bar{\partial}_{\overline{X}}\overline{e}}\,.
            \]
          Hence $TM\oplus E^*$ is isomorphic as a complex Lie algebroid to the Lie
          algebroid $T^{1,0}M\oplus (E^{0,1})^*$ which is induced by this connection,
          with bracket $([X_1,X_2],\nabla_{X_1}^*\eps_2- \nabla_{X_2}^*\eps_1)$ for
          $X_1,X_2\in\Gamma(T^{1,0})$ and $\eps_1,\eps_2\in\Gamma((E^{0,1})^*)$. Complexification
          of $\AAA$ and restriction to these eigenbundles gives precisely this bracket. 
      \end{example}
    
    \begin{example}
    		In the case where the generalised complex structure is induced by a 
    		linear symplectic structure $\omega^\flat\colon TE\to T^*E$, the side 
    		morphism is coming from an isomorphism $\tau\colon TM\to E^*$. The
    		fibrewise complex structure $j$ on $TM\oplus E^*$ is after identification
    		of $E^*$ with $TM$ via $\tau$ simply the complex structure of $T_\C M$, 
    		the anchor is then given by $\rho=\half\id_{T_\C M}$ and the bracket
    		$\AAA$ is given by $\half[\cdot,\cdot]_\C$. This follows from the
    		fact that $TM\oplus E^*$ is isomorphic as complex Lie algebroid to 
    		$U_+=\grap(-\tau_\C)$ according to Proposition \ref{prop_AAA_U_lie}, 
    		and the bracket of $\grap(-\tau)$ is shown in \cite{Jotz18a} to be 
    		the Lie bracket of vector fields. The factor of $\half$ comes
    		here from the isomorphism $TM\oplus E^*\to U_+$.
    	\end{example}

The description of VB-Dirac structures in $\TT E$ via adapted Dorfman 
connections of \cite{Jotz18a} can directly be extended to complex VB-Dirac 
structures and adapted complex Dorfman connections, by simply demanding 
complex linearity where appropriate. This leads to the following 
adaptation of a theorem in \cite{Jotz18a}.
\begin{thm}\label{thm_cplx_VB_Dirac}
	Let $D$ be a sub-double vector bundle of $\TT_\C E$ over $E$ and 
	$U\subseteq T_\C M\oplus E_\C^*$, with core 
	$K\subseteq E_\C\oplus T^*_\C M$ such that $D$ is a complex 
	subbundle of $\TT_\C E\to E$.
	Let $\Delta$ be a complex $(T_\C M\oplus E_\C ^*)$-Dorfman connection 
	on $E_\C \oplus T^*_\C M$ which is adapted to $D$. Then $D$ is a 
	complex VB-Dirac structure if and only if $U=K^\circ$ and 
	$(U,\pr_{T_\C M}|_U,\DMbra{\cdot}{\cdot}_\Delta|_U)$ is a complex
	Lie algebroid. 
\end{thm}

This description of complex VB-Dirac structures together with 
Theorem \ref{thm_lgcs_DMC} then leads to the following description
of linear generalised complex structures.
\begin{cor}\label{cor_lgcs_VB_Dirac}
	A linear generalised complex structure $\JJ$ on a vector bundle 
	$E$ is equivalent to a pair of transverse, complex conjugated 
	complex VB-Dirac structures $D_\pm$ in $\TT_\C E$.
\end{cor}

\section{Generalised complex structures on Lie algebroids}
\label{sec_gcLA}
Let $A\rightarrow M$ be a Lie algebroid with anchor
$\rho\colon A\to TM$.  In this case the generalised tangent bundle
$\TT A$ is itself a Lie algebroid over the side $TM\oplus A^*$, as
described for example in \cite{LiBland12} and \cite{Jotz18a}.  This
section considers a linear generalised complex structure on $A$, that
is also compatible with the Lie algebroid structure on $\TT A$.

\begin{mydef}\label{def_gc_LA}\cite{JoStXu16}
  A \textbf{generalised complex Lie algebroid} is a Lie algebroid
  $A\to M$ equipped with a linear generalised complex structure
  $\JJ\colon\TT A\rightarrow \TT A$ which is also a Lie
  algebroid morphism over the side morphism
  $j\colon TM\oplus A^*\to TM\oplus A^*$.
\end{mydef}

In the situation of the last definition, choose a Dorfman connection
$\Delta\colon \Gamma(TM\oplus A^*)\times\Gamma(A\oplus T^*M)\to
\Gamma(A\oplus T^*M)$ that is adapted to $\JJ$.  It follows directly
from the symmetry of the linear splitting that also the lift
$\slift\colon \Gamma(A)\to \Gamma_{TM\oplus A^*}^\ell(\TT A)$ is
compatible with $\JJ$, that is
\[
	\JJ\left(\slift(a)\right)=\slift(a)\circ j\
      \]
      for all $a\in\Gamma(A)$.
      
      Hence, by the results in \cite{DrJoOr15}, $\JJ$ is a Lie algebroid 
      morphism if and only if $(j,-j^t,0)$ is an automorphism of the
      2-representation 
      $\bigl((\rho,\rho^t),\nabla^\bas,\nabla^\bas,R_\Delta^\bas\bigr)$
      corresponding to the VB-algebroid structure on $TA\oplus T^*A$
      in the linear splitting defined by the adapted skew-symmetric Dorfman
      connection $\Delta$, see Theorem \ref{thm_TTA_LA}.
      That is,
      \begin{enumerate}
\item $j\circ(\rho,\rho^t)=-(\rho, \rho^t)\circ j^t$,
\item $\nabla_a^{\bas,\Hom}(j,-j^t)=0$ for all $a\in\Gamma(A)$, and
\item $-j^t\circ R_\Delta^\bas= R_\Delta^\bas\circ j $.
\end{enumerate}
Since the basic connections defined by a skew-symmetric Dorfman
connection are dual to each other, the second equality reduces to
$\nabla_a^\bas \circ j=j\circ\nabla_a^\bas$ for all $a\in\Gamma(A)$.
As observed before, $j$ and $j_C=-j^t$ are fibrewise complex structures 
on $TM\oplus A^*$ and $A\oplus T^*M$, respectively. The properties
above simply state that $(\rho,\rho^t)$, $\nabla_a^\bas$ and $R_\Delta^\bas$
are all complex linear.
Together with Theorem \ref{thm_lgcs_DMC} this 
immediately give the following characterisation of a generalised
complex Lie algebroid.
\begin{thm}\label{cor_lgcs_LA}
	Let $A$ be a Lie algebroid over $M$ with anchor $\rho$. A linear 
	generalised complex structure $(j,[\Delta])$ on $A$ (see Thm \ref{thm_lgcs_DMC}) 
	is compatible with the Lie algebroid structure in the sense of Definition 
	\ref{def_gc_LA} 	if and only if the basic connections and basic curvature 
	induced by any representative $\Delta$ in $[\Delta]$, as well as the 
	map $(\rho,\rho^t)$	are complex linear with respect to the complex 
	structures $j$ on $TM\oplus A^*$ and $-j^t$ on $A\oplus T^*M$.
      \end{thm}
		
A straightforward complex extension of the corresponding result
in \cite{Jotz19a} yields the following description of complex LA-Dirac
structures in terms of an adapted splitting. 

\begin{cor}\label{cor_cplx_VB_Dirac_LA}
	A complex VB-Dirac structure $D\subseteq\TT_\C A$ with side 
	$U\subseteq T_\C M\oplus A_\C^*$ and core 
	$K\subseteq A_\C\oplus T_\C^*M$ is additionally a Lie subalgebroid 
	of $\TT_\C A \to T_\C M\oplus A_\C^*$ if and only for an adapted 
	complex $(T_\C M\oplus A_\C^*)$-Dorfman connection $\Delta$ on 
	$A_\C \oplus T^*_\C M$ the following conditions are satisfied for 
	all $a,b\in\Gamma(A)$, $u\in\Gamma(U)$:
	\begin{enumerate}
		\item $(\rho,\rho^t)_\C (K)\subseteq U$, 
		\item $\nabla^\bas_a u\in\Gamma(U)$, 
		\item $R^\bas_\Delta(a,b)u\in\Gamma(K)$. 
	\end{enumerate}
	Note that here $\nabla^\bas$ and $R^\bas_\Delta$ denote the
	complex basic connection and complex basic curvature defined
	by complex linear extension of the formulas in the real case.
\end{cor}

The following description of generalised complex Lie algebroids in
terms of LA-Dirac structures
is an immediate consequence of this
description of complex LA-Dirac structures, together with Theorem
\ref{cor_cplx_VB_Dirac_LA} and Corollary \ref{cor_lgcs_VB_Dirac}.

\begin{cor}\label{cor_LA_Dirac}
	A generalised complex structure on a Lie algebroid is equivalent 
	to a pair of transverse, complex LA-Dirac structures in $\TT_\C A$. 
\end{cor}
\begin{proof} 
	According to  Corollary \ref{cor_lgcs_VB_Dirac} a linear generalised
	complex structure $\JJ$ on $A$ is equivalent to a pair of transversal, 
	complex VB-Dirac structures $D_\pm$. 
	Choose a Dorfman connection adapted to $\JJ$ as in Proposition
	\ref{prop_adapted_DMC}. Then the complexification of $\Delta$ is 
	adapted to $D_+$ and to $D_-$ simultaneously. According to 
	the considerations above, $\JJ$ is a Lie algebroid 
	morphism if and only if 
	\begin{enumerate}
		\item $(\rho,\rho^t)\circ j_C=j\circ (\rho,\rho^t)\,,$
		\item $\nabla_a^\bas \circ j=j\circ\nabla_a^\bas\,,$
		\item $j_C\circ R_\Delta^\bas (a,b)
			=R_\Delta^\bas (a,b)\circ j\,.$
	\end{enumerate}
	The first condition is equivalent to 
	$(\rho,\rho^t)(K_\pm)\subseteq U_\pm$, the second condition is
	equivalent to $\nabla^\bas_a u_\pm\in\Gamma(U_\pm)$ for any 
	$a\in\Gamma(A)$ and $u_\pm\in\Gamma(U_\pm)$ and the third condition
	is equivalent to $R^\bas_\Delta(a,b)(u_\pm)\in\Gamma(K_\pm)$ for
	all $a,b\in\Gamma(A)$ and $u_\pm\in\Gamma(U_\pm)$. 
	According to Theorem \ref{thm_cplx_VB_Dirac} these conditions 
	are equivalent to $D_\pm$ being complex LA-Dirac structures. 
\end{proof}

\subsection{The degenerate generalised complex structure on \texorpdfstring{$A\oplus T^*M$}{A+T*M}}
\label{sec_deg_gc_ca}

Recall that the vector bundle $A\oplus T^*M$ can be equipped with 
the structure of a degenerate Courant algebroid as described 
in \cite{Jotz19a}. The anchor is given by 
$\rho\circ\pr_A\colon A\oplus T^*M\to TM$, the (possibly degenerate) 
pairing and the bracket are given by
\begin{equation}\label{eq_deg_CA_pair_bra}
\begin{split}
	\bigpair{(a,\theta)}{(b,\eta)}_d
	:=&\pair{\rho(a)}{\eta}+\pair{\rho(b)}{\theta}\,,\\
	\bigDMbra{(a,\theta)}{(b,\eta)}_d
	:=&\bigl([a,b],\LL_{\rho(a)}\eta-i_{\rho(b)}\dd\theta\bigr)\,.
\end{split}
\end{equation}
where $a,b\in\Gamma(A)$ and $\theta,\eta\in\Gamma(T^*M)$.

This anchor, bracket and pairing satisfy all properties of 
a Courant algebroid except for the non-degeneracy of the pairing. 
As shown in \cite{Jotz19a}, the bracket can equivalently be described 
in terms of the Dorfman connection $\Delta$:
	\begin{equation}\label{eq_DCA_bra_DMC}
	\DMbra{\tau_1}{\tau_2}_d
	=\Delta_{(\rho,\rho^t)\tau_1}\tau_2
	-\nabla^\bas_{\pr_A \tau_2}\tau_1
      \end{equation}
       for 
$\tau_1,\tau_2\in\Gamma(A\oplus T^*M)$.

\begin{prop}\label{prop_jC_gcs_decCA}
	The core morphism $j_C\colon A\oplus T^*M \to A\oplus T^*M$
	of $\JJ$ satisfies $j_C^2=-\id$, is orthogonal with respect 
	to $\pair{\cdot}{\cdot}_d$ and the Nijenhuis torsion of 
	$j_C$ with respect to $\DMbra{\cdot}{\cdot}_d$ vanishes. 
	Hence, $j_C$ is a \textbf{degenerate generalised complex structure in
	the degenerate Courant algebroid $A\oplus T^*M$}.
\end{prop} 
\begin{proof}
  Recall from Proposition \ref{prop_J_gacs} that $j_C^2=-\id$ and
  $j=-j_C^t$. Together with the property
  $j\circ(\rho,\rho^t)=-(\rho,\rho^t)\circ j^t=(\rho,\rho^t)\circ
  j_C$, see Theorem \ref{cor_lgcs_LA}, it is easy to check that $j_C$
  is orthogonal with respect to the degenerate pairing.

  Theorem \ref{cor_lgcs_LA} gives as well the equality
  $\nabla_a^\bas \circ j_C=j_C\circ\nabla_a^\bas$. Using the formula
  \eqref{eq_DCA_bra_DMC} it is then easy to compute the Nijenhuis
  torsion of $j_C$:
	\[
		N_{j_C,\DMbra{\cdot}{\cdot}_d}(\tau_1,\tau_2)
		=\Delta_{(\rho,\rho^t)\tau_1}\tau_2
		-\Delta_{(\rho,\rho^t)j_C\tau_1}j_C\tau_2
		+j_C\Delta_{(\rho,\rho^t)j_C\tau_1}\tau_2
		+j_C\Delta_{(\rho,\rho^t)\tau_1}j_C\tau_2\,,
	\]
	which vanishes for any linear generalised complex 
	structure according to \eqref{eq_Nij_lin_core_dual} with 
	$\nu=(\rho,\rho^t)\tau_1$ and $\tau=\tau_2$.
      \end{proof}

      \begin{prop}
  Let $A\to M$ be a Lie algebroid.  The restriction of the degenerate
  Courant algebroid structure on $A_\C \oplus T^*_\C M$ induces a
  complex Lie algebroid structure on $K_\pm$.
\end{prop}
\begin{proof}
  According to Proposition
  \ref{prop_jC_gcs_decCA} the morphism $j_C$ is a generalised complex
  structure in $A\oplus T^*M$.  The vanishing of the Nijenhuis tensor
  and $\C$-linearity of the complexified bracket imply that the
  bracket restricts to the $\pm i$-eigenbundle $K_\pm$ of
  $j_{C,\C}$. Since $(\rho,\rho^t)_\C$ sends $K_\pm$ to
  $U_\pm=K_\pm^\circ$, the pairing restricted to $K_\pm$ vanishes and
  thus the restricted bracket is skew-symmetric and defines a Lie
  algebroid structure on $K_\pm$.
\end{proof}

\subsection{The complex \texorpdfstring{$A$}{A}-Manin pair}
\cite{Jotz19a} defines $A$-Manin pairs for a given Lie algebroid $A$
over $M$ and constructs an equivalence between $A$-Manin pairs and
Dirac bialgebroids over $A$. Again, the complex linear extension of
these results is straightforward.

\begin{mydef}
  Let $A\to M$ be a Lie algebroid.  A \textbf{complex $A$-Manin pair}
  consists of a complex Courant algebroid $C$ over $M$, a complex
  Dirac structure $U\rightarrow M$ in $C$, with a morphism
  $\iota\colon U\hookrightarrow T_\C M\oplus A_\C^*$ such that
  $\rho_U=\pr_{T_\C M}\circ\iota$ and a morphism of (degenerate)
  complex Courant algebroids
  $\Phi\colon A_\C\oplus T_\C^* M\rightarrow C$ such that
	\[
		\Phi(A_\C \oplus T_\C^* M)+U=C
	\] 
	and $\pair{u}{\Phi(\tau)}_C=\pair{\iota(u)}{\tau}$ 
	for all $(u,\tau)\in U\times_M (A_\C\oplus T_\C^*M)$.
\end{mydef}

\cite{Jotz19a} shows that the Courant algebroid structure on $C$ and
the morphism $\Phi$ can be recovered from the Lie algebroid structures
on $A$, $U$ and $\iota$. All the arguments can be extended complex
linearly to obtain the following straightforward consequence.

\begin{prop}\label{prop_Cpm}
  Let $U_\pm$ be the $\pm i$-eigenbundles of the side morphism $j$ of
  a generalised complex structure $\JJ$ on a Lie algebroid $A\to M$,
  and let $K_\pm=U_\mp^\circ$. Define
	\[
		C_{\pm}:=
		\frac{U_{\pm}\oplus(A_\C \oplus T_\C^*M)}
		{\grap\bigl(-(\rho,\rho^t)_\C\an{K_{\pm}}\bigr)
		}\,,
	\]
	and define an anchor map, a $\C$-bilinear pairing and a bracket
	as follows. For $u,u_1,u_2\in\Gamma(U_\pm)$, 
	$\tau,\tau_1,\tau_2\in\Gamma(A_\C\oplus T^*_\C M)$ 
	define the anchor by
	\[
		c_\pm(u\oplus \tau):=\rho_{U_\pm}(u)+(\rho_A)_\C\circ \pr_{A_\C}\tau\,,
	\]
	the pairing by
	\begin{equation}\label{eq_Cpm_pair}
		\pair{u_1\oplus \tau_1}{u_2\oplus \tau_2}_{C_\pm}
		:=\pair{u_1}{\tau_2}+\pair{u_2}{\tau_1}
		+\pair{\tau_1}{(\rho,\rho^t)_\C(\tau_2)}\,,
	\end{equation}
	and the bracket by
	\begin{equation}\label{eq_Cpm_bra}
	\begin{split}
	\DMbra{u_1\oplus \tau_1}{u_2\oplus \tau_2}_{C_\pm}
		&:=\Bigl([u_1,u_2]_{U_{\pm}}
		+\nabla^{\bas,\C}_{\pr_{A_\C} \tau_1}u_2
		-\nabla^{\bas,\C}_{\pr_{A_\C} \tau_2}u_1\Bigr)\\
		&\oplus\Bigl(\DMbra{\tau_1}{\tau_2}_{d,\C}
		+\Delta_{u_1}^{\C}\tau_2-\Delta_{u_2}^{\C}\tau_1
		+\bigl(0,\dd_\C\pair{\tau_1}{u_2}\bigr)\Bigr)\,.
	\end{split}
	\end{equation}
	Then $C_\pm$ are both complex Courant algebroids and 
	$(C_{\pm},U_{\pm})$ together with 
	$\iota\colon U_\pm\hookrightarrow T_\C M\oplus A_\C^*$ 
	and $\Phi\colon A_\C\oplus T_\C^* M\rightarrow C$ the canonical 
	inclusions are complex $A$-Manin pairs. 
\end{prop}
Recall that $U_+$ with its complex Lie algebroid structure is
isomorphic to the complex Lie algebroid
$(TM\oplus A^*, j, \rho_j,\AAA)$ (see Proposition
\ref{prop_AAA_U_lie}). Hence the result above realises the latter
Courant algebroid as a Dirac structure in the complex Courant
algebroid $C_+$.

Next, the generalised complex structure on $A$ induces generalised
complex structures $J_\pm$ in the Courant algebroids $C_\pm$ defined
by Proposition \ref{prop_Cpm}.
\begin{prop}
	Let $u\oplus \tau\in\Gamma(C_\pm)$. Then
	\[
		J_\pm(u\oplus \tau):=j_\C u\oplus j_{C,\C}\tau
	\]
	is well-defined and a generalised complex structure in $C_\pm$. 
\end{prop}
\begin{proof}
	Take any element $(-(\rho,\rho^t)_\C k)\oplus k$ of 
	$\grap\bigl(-(\rho,\rho^t)_\C\an{K_{\pm}}\bigr)$. Then
	\begin{equation*}
	J_\pm\bigl((-(\rho,\rho^t)_\C k)\oplus k\bigr)
	=\pm i\bigl((-(\rho,\rho^t)_\C k)\oplus k\bigr)\,,
	\end{equation*}
	which is again an element of 
	$\grap(-(\rho,\rho^t)\an{K_{\pm}})$. Thus the map $J_\pm$
	is well-defined on $C_\pm$. 
It is clear that $J_\pm^2=-1$. Orthogonality follows from an 
	easy computation using $j^2=-1$, $j_\C^t=-j_{C,\C}$ 
	and 	$(\rho,\rho^t)_\C\circ j_{C,\C}=j_\C\circ(\rho,\rho^t)_\C$.
	
	The last remaining condition is the vanishing of the Nijenhuis 
	torsion of $J_\pm$ with respect to the bracket on $C_\pm$.
	Lengthy, but straightforward computations making use of the previously
	proven facts that $\nabla^{\bas,\C}$ preserves $U_\pm$, $\Delta^\C_u$ 
	preserves $K_\pm$ and that the Nijenhuis torsion of $j_C$ with
	respect to $\DMbra{\cdot}{\cdot}_d$ vanishes (Proposition 
	\ref{prop_jC_gcs_decCA}), establish this condition. Hence $J_\pm$ 
	defines a generalised complex structure in $C_\pm$. 
\end{proof}

\subsection{The Lie bialgebroid \texorpdfstring{$(U_\pm,K_\mp)$}{(U,K)}; proof of Theorem \ref{main2}}
\label{sec_Drinfeld}

This section shows that the pair $(U_\pm,K_\mp)$ forms a Lie
bialgebroid with Drinfeld double Courant algebroid isomorphic to
$C_\pm$. First, observe that the identity $U_\pm^\circ=K_\pm$ induces
isomorphisms $U_\pm^*\cong K_\mp$ and $K_\pm^*\cong U_\mp$.  The
following theorem establishes then Theorem \ref{main2}, since $U_+$
with its complex Lie algebroid structure is isomorphic to the complex
Lie algebroid $(TM\oplus A^*, j, \rho_j,\AAA)$ (see Proposition
\ref{prop_AAA_U_lie}).
\begin{thm}\label{thm_iso_CA}
  Let $(A\to M, \JJ)$ be a generalised complex Lie algebroid. There is
  an isomorphism of vector bundles
	\begin{equation}\label{eq_iso_CA}
	\begin{split}
		F\colon U_\pm\oplus K_\mp&\to C_\pm, \qquad 
		(u,k) \mapsto u\oplus k.
	\end{split}
	\end{equation}
	This equips $U_\pm\oplus K_\mp$ with the structure of a
        Courant algebroid, in which the complex Lie algebroids $U_\pm$
        and $ K_\mp$ are transversal Dirac structures. Thus the pair
        $(U_\pm,K_\mp)$ is a complex Lie bialgebroid. $F$ is an
        isomorphism of Courant algebroids where $U_\pm\oplus K_\mp$ is
        the Drinfeld double Courant algebroid of this Lie bialgebroid.
\end{thm}
\begin{proof}
	It is easy to verify that 
	\begin{equation*}
		u\oplus \tau \mapsto 
		\Bigl(u+\half(\rho,\rho^t)_\C\bigl(\tau\mp ij_\C\tau\bigr),
		\half(\tau\pm ij_\C\tau)\Bigr)\,.
	\end{equation*}
	is well-defined and defines an inverse to $F$.
	
	The Courant algebroid structure of $C_\pm$ induces via this
        isomorphism a Courant algebroid structure on the bundle
        $U_\pm\oplus K_\mp$. The following shows that the Lie
        algebroids $U_\pm$ and $K_\mp$ are Dirac structures in
        $C_\pm$. Liu, Weinstein and Xu showed in \cite{LiWeXu97} that
        two transversal Dirac structures in a Courant algebroid are
        equivalent to a Lie bialgebroid. Thus $(U_\pm,K_\mp)$ is a Lie
        bialgebroid, which induces the Drinfeld double Courant
        algebroid on $U_\pm\oplus K_\mp$. It remains then to show that
        the pairing and bracket of $C_\pm$ are equal to the pairing
        and bracket of this Drinfeld double and that they are thus
        isomorphic as Courant algebroids with the isomorphism given by
        the map $F$ defined in \eqref{eq_iso_CA}.

	With the definition of the bracket in $C_\pm$ in \eqref{eq_Cpm_bra},
	it follows for $k_1,k_2\in\Gamma(K_\mp)$ directly that
	\begin{align*}
	\DMbra{(0\oplus k_1)}{(0\oplus k_2)}_{C_\pm}
	=0\oplus \DMbra{k_1}{k_2}_{d,\C}\,.
	\end{align*}
	Similarly, for two sections $u_1,u_2\in\Gamma(U_\pm)$, the bracket is 
	$\DMbra{(u_1\oplus 0)}{(u_2\oplus 0)}_{C_\pm}=[u_1,u_2]_{U_\pm}\oplus 0$. 
	
	From the definition of the pairing in $C_\pm$ in \eqref{eq_Cpm_pair} 
	it is easy to see that both $U_\pm\oplus 0$ and $0\oplus K_\mp$ 
	are maximally isotropic with respect to this pairing and thus 
	Dirac structures in $C_\pm$. Thus by the argument in \cite{LiWeXu97}
	$(U_\pm,K_\mp)$ form complex Lie bialgebroids. 
	
	The anchor and pairing are easily seen to be equal to the anchor
	and the pairing in the Drinfeld double Courant algebroid. The 
	bracket in the Drinfeld double is defined for 
	$u_1,u_2\in\Gamma(U_\pm)$ and $k_1,k_2\in\Gamma(K_\mp)$ by
	\[
		\DMbra{(u_1,k_1)}{(u_2,k_2)}
		=\Bigl([u_1,u_2]+\LL_{k_1}^Ku_2-\iota_{k_2}\dd_K u_1,
		\quad [k_1,k_2]+\LL^U_{u_1}k_2-\iota_{u_2}\dd_U k_1 \Bigr)\,.		
	\]
	It only remains to be shown that the brackets of elements of
	the form $(u,0)$ with $(0,k)$ coincide, the rest follows by
	bilinearity, since the brackets on $U_\pm$ and $K_\mp$ were
	already shown to be inherited from the bracket in $C_\pm$. A 
	straightforward computation using \eqref{eq_DCA_bra_DMC}
	shows 
	\[
		\DMbra{u\oplus 0}{0\oplus k}_{C_\pm}=
		-\nabla^{\bas,\C}_{\pr_{A_\C}k}u \oplus \Delta^\C_u k=
		-\iota_k \dd_K u \oplus \LL^U_{u}k\,,
              \]
              see also \cite{Heuer19} for details.
	Hence $F$ defines indeed an isomorphism of Courant algebroids
	from $C_\pm$ to the Drinfeld double $U_\pm\oplus K_\mp$. 
\end{proof}

\begin{example}\label{ex_cplx_LBA_CA}  
       In the situation of Example \ref{hol_ex_cpla},
if $A\to M$ is equipped with a Lie algebroid structure and a
compatible linear complex structure, then the eigenbundles are Lie
algebroids and thus also define Drinfeld double Courant algebroids
\begin{equation*}
\begin{split}
	C_T^{1,0}&=T^{1,0}M\oplus (T^{1,0}M)^*\,,\qquad
	C_A^{1,0}=A^{1,0}\oplus (A^{1,0})^*\,,\\
	C_T^{0,1}&=T^{0,1}M\oplus (T^{0,1}M)^*\,,\qquad
	C_A^{0,1}=A^{0,1}\oplus (A^{0,1})^*\,,\\
\end{split}
\end{equation*}
induced by the Lie bialgebroid structure where $T^*_\C M$ and $A_\C^*$
are endowed with trivial Lie algebroid structures. That is, the
brackets on $C_T^{1,0}$ and $C_T^{0,1}$ are given by
\[
	\DMbra{(X,\theta)}{(Y,\eta)}=([X,Y],\LL_X\eta-\iota_Y\dd\theta)\,,
\]
and analogously on $C_A^{1,0}$ and $C_A^{0,1}$. \eqref{eq_cplx_U_K}
shows that as vector bundles $C_+=C_T^{1,0}\oplus C_A^{0,1}$ and
$C_-=C_T^{0,1}\oplus C_A^{1,0}$. The computations in \cite{Heuer19}
show that these are orthogonal decompositions with respect to the pairings 
in $C_\pm$ and that the brackets in $C_T^{1,0}$, $C_T^{0,1}$, $C_A^{1,0}$ 
and $C_A^{0,1}$ coincide with the respective 
restrictions of the brackets in $C_\pm$. In other words, they form
matched pairs of Courant algebroids, a notion introduced in \cite{GrSt14}.
\end{example}

\section{Generalised complex structures in VB-Courant algebroids}
\label{sec_gcs_VB_CA}

In this section the results of Section \ref{sec_gcs_vb} are extended to 
general VB-Courant algebroids $(\E;Q,B;M)$. This leads to a definition
of generalised complex structures in split Lie 2-algebroids.

A linear splitting $\Sigma$ of the double vector bundle $\E$ is called
\textbf{Lagrangian} if the image of $\Sigma$ is isotropic in $\E$.
The paper \cite{Jotz19b} shows that a change of Lagrangian splittings
corresponds to a skew-symmetric element
$\Phi_{12}\in\Gamma(Q^*\otimes B^*\otimes Q^*)$.

Only the description of linear splittings with Dorfman connections
relies on the special case of $TE\oplus T^*E$. The other results of 
Section \ref{sec_gcs_vb} only use the abstract structure of a metric 
double vector bundle and Lagrangian lifts. They therefore generalise
to VB-Courant algebroids in the following way. 

Fix a Lagrangian splitting $\Sigma$ of $\E$  and denote the 
corresponding lift by $\sigma\colon \Gamma(Q)\to \Gamma_B^\ell(\E)$. 
Consider a double vector bundle morphism $\JJ\colon \E\to \E$ 
over $\id_B$ and $j\colon Q\to Q$ with core morphism 
$j_C\colon Q^*\to Q^*$. As in Lemma
\ref{lem_J_lift}, the following definition of $\Phi$ depends on the 
choice of the splitting.
\begin{lemma}\label{lem_J_lift_VB_CA}
	Given a double vector bundle morphism $\JJ\colon \E\to \E$ over $j$ and 
	$\id_B$	there is $\Phi\in\Gamma(Q^*\otimes B^*\otimes Q^*)$ defined by 
	setting for any $q\in\Gamma(Q)$ 
	\[
		\JJ(\sigma(q))=\sigma(jq)+\corelinear{\Phi(q)}\,.
	\]
\end{lemma}
Furthermore, the following lemmas generalise the description of generalised almost 
complex structures on a vector bundle in Section \ref{sec_gcs_vb}.
\begin{lemma}\label{lem_J_squ_VB_CA}
	A double vector bundle morphism $\JJ\colon \E\to\E$  satisfies
	$\JJ^2=-\id_{TE\oplus T^*E}$ if and only if for any Lagrangian splitting
	and corresponding $\Phi$, and for any $q\in\Gamma(Q)$:
	\begin{enumerate}
		\item $j^2=-\id_{Q}$\,,
		\item $j_C^2=-\id_{Q^*}$\,, 
		\item $\Phi(j(q))=-j_C\circ(\Phi(q))$\,.
	\end{enumerate}
\end{lemma}
\begin{lemma}
	\label{lem_J_ort_VB_CA}
	A double vector bundle morphism $\JJ\colon \E\to\E$ such that 
	additionally $\JJ^2=-1$, is orthogonal if and only if for any 
	Lagrangian splitting 
	\begin{enumerate}
		\item $j=-(j_C)^t$\,, 
		\item $\pair{j(q_1)}{\Phi(q_2)(b)}
		=-\pair{j(q_2)}{\Phi(q_1)(b)}$
              \end{enumerate}
              for all $b\in \Gamma(B)$ and 
              $q_1,q_2\in\Gamma(Q)$.
\end{lemma}
Now define a 2-form $\Psi\in\Omega^2(Q,B^*)$ by setting
$\Psi(q_1,q_2):=\Phi(q_1)^t(q_2)$. This definition yields the
following, as in Proposition \ref{prop_J_gacs}.
\begin{prop}\label{prop_J_gacs_VB_CA}
	A morphism $\JJ\colon \E\to\E$ is a generalised almost complex 
	structure in $\E$, if and only if for any Lagrangian splitting 
	\begin{enumerate}
		\item $j^2=-1$\,,
		\item $j=-(j_C)^t$\,,
		\item $\Psi$ is skew-symmetric, that is $\Psi\in\Omega^2(Q,B^*)$\,,
		\item $\Psi(q_1,q_2)=-j^*\Psi(q_1,q_2)$ for $q_1,q_2\in\Gamma(Q)$.
	\end{enumerate}
\end{prop}
Also in this case, a Lagrangian splitting can be adapted to the
generalised almost complex structure. As mentioned before it was shown
in \cite{Jotz19b} that such a change of splittings corresponds to a
skew-symmetric element
$\Phi_{12}\in\Gamma(Q^*\otimes B^*\otimes Q^*)$. 
\begin{prop}\label{prop_adapted_lift_VB_CA}
	Given a generalised almost complex structure $\JJ$ in a VB-Courant 
	algebroid	$(\E;Q,B;M)$ with side morphism $j\colon Q\to Q$, there is 
	a Lagrangian lift $\sigma\colon \Gamma(Q)\to \Gamma_B^\ell(\E)$, such 
	that for any $q\in\Gamma(Q)$
	\[
		\JJ(\sigma(q))=\sigma(jq)\,.
	\]
\end{prop}
\begin{proof}
	Fix any Lagrangian lift $\sigma_1$ of $\E$. This defines by 
	Lemma \ref{lem_J_lift_VB_CA} a tensor $\Phi_1\in\Gamma(Q^*\otimes B^*\otimes Q^*)$. 
	Define another tensor $\Phi_{12}\in\Gamma(Q^*\otimes B^*\otimes Q^*)$ by 
	setting for any $q\in\Gamma(Q)$ and $b\in\Gamma(B)$
	\[
		\Phi_{12}(q)(b):=\half j_C(\Phi_1(q)(b))\,.
	\]
	By Lemma \ref{lem_J_squ_VB_CA} and Lemma \ref{lem_J_ort_VB_CA}, 
	$\Phi_{12}$ is skew-symmetric. Define a new Lagrangian lift by 
	$\sigma_2(q):=\sigma_1(q)-\corelinear{\Phi_{12}(q)}$. This lift
	satisfies the desired property.
\end{proof}

Using this existence of an adapted Lagrangian splitting, use the 
correspondence of VB-Courant algebroid structures to split Lie 2-algebroids
proved in \cite{Jotz19b}. Fix such an adapted Lagrangian splitting
as in Proposition \ref{prop_adapted_lift_VB_CA}. Then the VB-Courant algebroid
structure is equivalent to a split Lie 2-algebroid structure 
$(\rho_Q,\partial_B^t,\DMbra{\cdot}{\cdot}_\Delta,\nabla,\omega)$ on $Q\oplus B^*$, 
where the bracket in $\E$ is described by the dull bracket on $Q$ and the 
dual Dorfman connection as follows.
\begin{align*}
	\DMbra{\sigma(q_1)}{\sigma(q_2)}
	&=\sigma(\DMbra{q_1}{q_2}_\Delta)-\corelinear{R_\omega(q_1,q_2)}\\
	\DMbra{\sigma(q)}{\tau^\dagger}&=(\Delta_q\tau)^\dagger \quad \text{ and }\quad
	\DMbra{\tau_1^\dagger}{\tau_2^\dagger}=0\,.
\end{align*}
Here $R_\omega(q_1,q_2):=\omega(q_1,q_2,\cdot)^t\in\Gamma(\Hom(B,Q^*))$

This description of the Courant algebroid bracket yields similar 
computations and results for the Nijenhuis tensor of core sections and lifts 
for a linear generalised almost complex structure in the VB-Courant algebroid
$\E$ as in Section \ref{sec_lgcs_int} in the special case of $TE\oplus T^*E$. 

First, analogously to the computations in Section \ref{sec_lgcs_int},
the section $N_\JJ(\sigma(q),\tau^\dagger)$ vanishes for any $q\in\Gamma(Q)$ and $\tau\in\Gamma(Q^*)$ if and 
only if $N_{j,\DMbra{\cdot}{\cdot}_\Delta}$ vanishes. Second, the analogous 
computation for the Nijenhuis tensor of two lifts gives 
\[
	\begin{split}
	N_\JJ(\sigma(q_1),\sigma(q_2))&=
	\slift(N_{j,\DMbra{\cdot}{\cdot}_Q}(q_1,q_2))
	+\corelinear{R_\omega(j(q_1),j(q_2))}
	-\corelinear{R_\omega(q_1,q_2)}\\
	&\quad-\corelinear{j_C\circ R_\omega(j(q_1),q_2)}
	-\corelinear{j_C\circ R_\omega(q_1,j(q_2))}
	\end{split}
\]

Dualising the property 
\[
	R_\omega(j(q_1),j(q_2))		
	-R_\omega(q_1,q_2)
	-j_C\circ R_\omega(j(q_1),q_2)
	-j_C\circ R_\omega(q_1,j(q_2))=0\,,
\]
by evaluating at any $b\in\Gamma(B)$ and then pairing with $q_3$ gives 
as an equivalent condition on $\omega\in\Omega^3(Q,B^*)$ the following: 
\[
	\omega(q_1,q_2,q_3)		
	-\omega(jq_1,jq_2,q_3)
	-\omega(jq_1,q_2,jq_3)
	-\omega(q_1,jq_2,jq_3)=0\,.
\]
This yields the following proposition.
\begin{prop}
	A linear generalised almost complex structure $\JJ$ in $\E$ over 
	$j\colon Q\to Q$ is integrable if and only if for any adapted 
	Lagrangian splitting
of the corresponding split Lie 2-algebroid,	
	\begin{enumerate}
		\item $N_{j,\DMbra{\cdot}{\cdot}_\Delta}(q_1,q_2)=0$\,,
		\item $\omega(q_1,q_2,q_3)		
				-\omega(jq_1,jq_2,q_3)
				-\omega(jq_1,q_2,jq_3)
				-\omega(q_1,jq_2,jq_3)=0$
                              \end{enumerate}
                              for any $q_1,q_2,q_3\in\Gamma(Q)$.
\end{prop}

As before the vector bundle morphism $j\colon Q\to Q$ defines an
equivalence relation on the Lagrangian splittings.
\begin{mydef}
	Given a VB-Courant algebroid $(\E;Q,B;M)$ and a vector bundle 
	morphism $j\colon Q\to Q$, two Lagrangian splittings $\Sigma_1$ and
	$\Sigma_2$ are $j$-\textbf{equivalent} if the corresponding change
	of splittings $\Psi\in\Omega^2(Q,B^*)$ satisfies 
	$\Psi(q_1,q_2)=\Psi(jq_1,jq_2)$ for any 
	$q_1,q_2\in\Gamma(Q)$.
\end{mydef}
Analogously to Lemma \ref{lem_j_equiv}, given a splitting $\Sigma_1$
which is adapted to a linear generalised almost complex structure
$(\JJ,j)$, then a second splitting $\Sigma_2$ is also adapted to
$(\JJ,j)$ if and only if $\Sigma_1$ and $\Sigma_2$ are
$j$-equivalent. This allows a formulation of the analogue of Theorem
\ref{thm_lgcs_DMC} in the general case.
\begin{thm}\label{thm_lgcs_VB_CA}
	A linear generalised complex structure $\JJ$ in a VB-Courant algebroid 
	$\E$ is equivalent to a vector bundle morphism $j\colon Q\to Q$ and a 
	$j$-equivalence class of linear splittings such that in the corresponding
	split Lie 2-algebroid $(\rho_Q,\partial_B^t,\DMbra{\cdot}{\cdot}_\Delta,\nabla,\omega)$ 
	over $Q\oplus B^*$ 
	\begin{enumerate}
		\item $j^2=-\id_Q$\,,
		\item $N_{j,\DMbra{\cdot}{\cdot}_\Delta}=0$\,,
		\item $\omega(q_1,q_2,q_3)		
				-\omega(jq_1,jq_2,q_3)
				-\omega(jq_1,q_2,jq_3)
				-\omega(q_1,jq_2,jq_3)=0$
                              \end{enumerate}
                               for any $q_1,q_2,q_3\in\Gamma(Q)$.
\end{thm}
Analogously to the case of $TE\oplus T^*E$,  a bracket $\AAA$ on $\Gamma(Q)$ 
can be defined by 
$\AAA(q_1,q_2)=\half\bigl(\DMbra{q_1}{q_2}_\Delta-\DMbra{jq_1}{jq_2}_\Delta\bigr)\,.$
The vanishing of $N_{j,\DMbra{\cdot}{\cdot}_\Delta}$ is equivalent to complex
bilinearity of $\AAA$ and the condition on $\omega$ in Theorem \ref{thm_lgcs_VB_CA}
implies the Jacobi identity for $\AAA$. This defines a complex Lie algebroid 
$(Q,\rho,\AAA$ with the complex anchor $\rho\colon Q\to T_\C M$ given by 
\[\rho(q)=\half\bigl(\rho_Q(q)-ij\rho_Q(q)\bigr)\,.\] 
But if the core-anchor $\partial_B$ is not surjective, then the 
condition on $\omega$ in Theorem \ref{thm_lgcs_VB_CA} is stronger
than the Jacobi identity of this bracket, since  
$\Jac_{\DMbra{\cdot}{\cdot}_\Delta}=\partial_B^t\circ \omega$. 
Therefore -- unlike in the special case of $TE\oplus T^*E$ -- here the 
complex Lie algebroid structure is not sufficient to describe 
the conditions on the linear generalised complex structure. 

Theorem \ref{thm_lgcs_VB_CA} suggests that a generalised complex
structure in a split Lie 2-algebroid should be defined as a tuple of maps
$(\rho_Q,\partial_B^t,\DMbra{\cdot}{\cdot},\nabla,\omega)$ over
$Q\oplus B^*$ is a vector bundle morphism $j\colon Q\to Q$, such that
for any $q_1,q_2,q_3\in\Gamma(Q)$ 
	\begin{enumerate}
		\item $j^2=-\id_Q$\,,
		\item $N_{j,\DMbra{\cdot}{\cdot}}=0$\,,
		\item $\omega(q_1,q_2,q_3)		
				-\omega(jq_1,jq_2,q_3)
				-\omega(jq_1,q_2,jq_3)
				-\omega(q_1,jq_2,jq_3)=0$\,.
	\end{enumerate}

        \appendix

      \section{Relation with the adapted generalised connections in \texorpdfstring{\cite{CoDa19}}{CD19}}\label{app_comparison}
      The equality $\JJ(\slift(\nu))=\slift(j\nu)$ in Proposition
      \ref{prop_adapted_DMC} for $\nu\in\Gamma(TM\oplus E^*)$ is
      equivalent to
\begin{equation}\label{geometric_adapted1} \JJ(L_\Delta)=L_\Delta
\end{equation}
for the horizontal space $L_\Delta\subseteq TE\oplus T^*E$
corresponding to $\Delta$.

Let $\mathbb E\to M$ be an arbitrary Courant algebroid. A
\emph{generalised connection} on $\mathbb E$ is a linear connection
$\nabla\colon\Gamma(\mathbb E)\times\Gamma(\mathbb E)\to\Gamma(\mathbb
E)$, which is compatible with the pairing in $\E$ (\cite{Gualtieri10}).  For 
instance, if $\nabla\colon\mx(M)\times\Gamma(\mathbb E)\to\Gamma(\mathbb E)$ 
is an ordinary metric linear connection, then
$\nabla^\rho\colon\Gamma(\mathbb E)\times\Gamma(\mathbb
E)\to\Gamma(\mathbb E)$ defined by
$\nabla^\rho_ee'=\nabla_{\rho(e)}e'$ for $e,e'\in\Gamma(\mathbb E)$,
is a generalised connection. 

Let $\JJ\colon\mathbb E\to\mathbb E$ be a generalised almost complex
structure.  The paper \cite{CoDa19} shows that there exists a metric 
linear connection
$\nabla\colon\mx(M)\times\Gamma(\mathbb E)\to\Gamma(\mathbb E)$ that
is \emph{adapted} to $\JJ$:
\begin{equation}\label{algebraic_adapted1}
  \nabla_\cdot \JJ =0.
\end{equation}
The pullback $\nabla^\rho$ is then a generalised connection
\emph{adapted} to $\JJ$ and its intrinsic torsion relative to the
connection is studied in \cite{CoDa19} in relation with the
integrability of $\JJ$ -- generalising the fact that an almost complex
structure $J\colon TM\to TM$ on a smooth manifold $M$ is integrable if
and only if there exists a complex-linear torsion-free connection
$\nabla\colon\mx(M)\times\mx(M)\to\mx(M)$.  The condition
\eqref{algebraic_adapted1} is equivalent to the generalised complex
structure $T\mathcal J\colon T\mathbb E\to T\mathbb E$ over $TM$
preserving the horizontal space $H_\nabla\subseteq T\mathbb E$ defined
by $\nabla$:
\begin{equation}\label{geometric_adapted2} 
  T\mathcal J(H_\nabla)=H_\nabla.
\end{equation}

The notion of adapted generalised connection in \cite{CoDa19} seems in
general different from the notion of adapted Dorfman connection in
Proposition \ref{prop_adapted_DMC}. However, as the similarity of
\eqref{geometric_adapted1} with \eqref{geometric_adapted2} suggests,
they are equivalent at least in a special situation, which is
explained in the remainder of this section.

\medskip Let $\mathbb E\to M$ be a Courant algebroid and denote the co-anchor 
of $\E$ by $\rho^*$, which is defined by composing $\rho^t$ with the isomorphism 
between $\E$ and $\E^*$ induced by the pairing in $\E$. Consider a
generalised connection
$\nabla\colon\Gamma(\E)\times\Gamma(\mathbb E)\to\Gamma(\mathbb E)$ such that
$\nabla_{\rho^* \theta}=0$ for all $\theta\in \Omega^1(M)$. It is
easy to see that
  \begin{equation}\label{eq_gconn_dconn}
    \Delta_{e_1}e_2=\llb e_1,e_2\rrb+\nabla_{e_2}e_1
  \end{equation}
  for all $e_1,e_2\in\Gamma(\mathbb E)$, defines a Dorfman connection
  $\Delta\colon\Gamma(\mathbb E)\times\Gamma(\mathbb
  E)\to\Gamma(\mathbb E)$, see also \cite{Jotz20}.  Conversely a Dorfman
  $\mathbb E$-connection on $\mathbb E$ defines a generalised connection 
  by \eqref{eq_gconn_dconn}, that must satisfy
  \begin{equation}\label{condition_eq_dorfman_linear}
    \nabla_{\rho^*\theta}e=-\llb
    e,\rho^*\theta\rrb+\Delta_e\rho^*\theta=\rho^*(-\ldr{\rho(e)}\theta+\ldr{\rho(e)}\theta)=0
  \end{equation}
  for all $e\in\Gamma(\mathbb E)$ and all $\theta\in
  \Omega^1(M)$. Hence Dorfman $\mathbb E$-connections on $\mathbb E$
  are equivalent with linear $\mathbb E$-connections on $\mathbb E$
  satisfying \eqref{condition_eq_dorfman_linear}.  In particular,
  since $\rho\circ \rho^*=0$ (see \cite{Roytenberg02}), the pullbacks
  of $TM$-connections are equivalent to a class of Dorfman
  $\mathbb E$-connections on $\mathbb E$. A metric $TM$-connection on
  $\mathbb E$ is equivalent to a Lagrangian splitting of the tangent prolongation
  of $\mathbb E$, which is a VB-Courant algebroid. The induced dull
  bracket on $\Gamma(\mathbb E)$ is the degree $1$ part of the splitting of the
  corresponding Lie $2$-algebroid, see \cite{Jotz19b}.

\medskip

Equations \eqref{geometric_adapted1} and \eqref{geometric_adapted2}
can in fact be related in the case of the standard Courant algebroid
$TM\oplus T^*M$ over a smooth manifold $M$.
 A computation shows that the canonical isomorphim

 \begin{equation}\label{can_iso_courant} 
     	{\footnotesize	
 	\begin{tikzcd}
                  T(TM\oplus T^*M) \ar[rr,"\mathcal I"]\ar[rd]\ar[dd] & & T(TM)\oplus T^*(TM) \ar[dd]\ar[rd] & \\
                  & TM  \ar[rr,"\id", near start, crossing over] & & TM \ar[dd] \\
                  TM\oplus T^*M \ar[rr,"\id", near start]\ar[rd] & & TM\oplus T^*M\ar[rd] & \\
                  & M\ar[rr,"\id"] \ar[uu, crossing over, leftarrow] &
                  & M
		\end{tikzcd}}
              \end{equation}
        arising from the canonical involution $I\colon TTM\to TTM$ 
        (see e.g.~\cite{JoStXu16} and references therein) sends
        $H_{\nabla}\subseteq T(TM\oplus T^*M)$ to
        $L_\Delta\subseteq T(TM)\oplus T^*(TM)$, if and only if $\Delta$ and
        $\nabla^\rho$ are related by \eqref{eq_gconn_dconn}.

Consider an almost
complex structure $J\colon TM\to TM$, as well as a torsion-free linear connection
$\nabla\colon\mx(M)\times\mx(M)\to\mx(M)$.  Consider the $TM$-connection
$\tilde\nabla\colon \mx(M)\times\Gamma(TM\oplus
T^*M)\to\Gamma(TM\oplus T^*M)$,
$\tilde\nabla_X(Y,\theta)=(\nabla_XY,\nabla^*_X\theta)$. Let
$\JJ_J$ be the generalised almost complex structure defined as in Example
\ref{ex_gcs_cplx}. The generalised connection $\tilde\nabla^\rho$ satisfies
\[\tilde\nabla^\rho_\cdot\JJ_J=0, \text{ or in other words } T\JJ_J(H_{\tilde\nabla})=H_{\tilde\nabla},
\]
if and only if  
  $\nabla_\cdot J=0$, i.e.~if and only  if $TJ(H_\nabla)=H_\nabla$.

  Here, an easy computation shows that $\Delta$ and
  $\tilde \nabla^\rho$ are related by \eqref{eq_gconn_dconn} if and
  only if $\Delta$ is the standard Dorfman connection defined by
  $\nabla$ as in Example \ref{def_std_DMC}. As a consequence
  \[\mathcal I(H_{\tilde\nabla})=L_\Delta.
  \]
  The canonical isomorphism
  $\mathcal I$ also transforms $T\JJ_J$ into the linear generalised almost 
  complex structure $\JJ_{J_T}$, where $J_T$ is the almost complex
  structure $I\circ TJ\circ I$ on the vector bundle $TM$ seen as a
  manifold. Then $\tilde\nabla_\cdot\JJ_J=0$ if and only if
        \[ \JJ_{J_T}(L_{\Delta})=L_{\Delta},
        \]
        where
        $\Delta\colon\Gamma(TM\oplus T^*M)\times\Gamma(TM\oplus
        T^*M)\to\Gamma(TM\oplus T^*M)$ is the standard Dorfman
        connection defined by $\nabla$ as in Example
        \ref{def_std_DMC}.  That is,
        \[ T\JJ_J(H_{\tilde\nabla})=H_{\tilde\nabla} \qquad \text{ if
            and only if } \qquad
          \JJ_{J_T}(L_{\Delta})=L_{\Delta}.
          \]
          In other words, $\tilde\nabla$ is adapted to $\JJ_J$ in the
          sense of \cite{CoDa19} if and only if $\Delta$ is
          adapted to $\JJ_{I\circ TJ\circ I}$ in the sense of
          Proposition \ref{prop_adapted_DMC}.

          More generally, let
          $\mathcal J\colon TM\oplus T^*M\to TM\oplus T^*M$ be a
          generalised almost complex structure and let
          $\nabla\colon\Gamma(TM)\times\Gamma(TM\oplus T^*M)\to
          \Gamma(TM\oplus T^*M)$ be a linear connection adapted to
          $\mathcal J$. As before
          $T\mathcal J\colon T(TM\oplus T^*M)\to T(TM\oplus T^*M)$ is
          a linear generalised almost complex structure over the
          identity on the base $TM$, and
          $\mathcal J\colon TM\oplus T^*M\to TM\oplus T^*M$ on the
          side. The isomorphism $\mathcal I$ in
          \eqref{can_iso_courant} tranforms $T\mathcal J$ into a
          linear generalised almost complex structure
          $\mathcal J_{TM}\colon T(TM)\oplus T^*(TM)\to T(TM)\oplus
          T^*(TM)$ in the standard VB-Courant algebroid over $TM$.
          Then since $\nabla$ is adapted to $\mathcal J$:
          \begin{equation*}
            T\JJ(H_{\nabla})=H_{\nabla},
          \end{equation*}
          which is again equivalent to
          \[ \JJ_{TM}(L_\Delta)=L_\Delta,\] where $\nabla^\rho$ and
          $\Delta$ are equivalent via \eqref{eq_gconn_dconn}.

\end{document}